\documentclass{article}
\usepackage[margin=2cm]{geometry}

\usepackage{import}

\usepackage{tikz}
\usepackage{pgfplots}
\usetikzlibrary{calc, spy}

\usepackage{amsmath, amssymb, bbold, mathtools}
\usepackage[ruled, vlined]{algorithm2e}
\usepackage{multirow}
\usepackage{graphicx}
\usepackage{tcolorbox}
\usepackage[]{hyperref}
\usepackage{bbding}
\usepackage[export]{adjustbox}
\usepackage{subcaption}

\usepackage{wasysym}
\usepackage{lmodern} 

\usepackage{authblk}

\usepackage{natbib}
 \bibpunct[, ]{(}{)}{,}{a}{}{,}

\usepackage{tikz}
\usetikzlibrary{calc}
\usetikzlibrary{shapes}

\usepackage{amsthm}
\usepackage[capitalise]{cleveref}
\theoremstyle{plain}
\newtheorem{theorem}{Theorem}
\newtheorem{lemma}{Lemma}
\newtheorem{corollary}{Corollary}

\newtheorem{proposition}{Proposition}

\usepackage{acronym}
\acrodef{wlog}[WLOG]{without loss of generality}
\acrodef{lsc}[lsc]{lower semi-continuous}

\definecolor{red}{RGB}{163, 31, 52}
\definecolor{gray}{RGB}{194, 192, 191}
\definecolor{blue}{RGB}{59, 89, 152}
\definecolor{green}{RGB}{0, 179, 0}

\newcommand{\norm}[1]{\left\|#1\right\|}

\newcommand{\bigO}[1]{\mathcal{O}\left(#1\right)}
\renewcommand{\tfrac}[2]{{#1}/{#2}}

\newcommand{\set}[2]{\left\{ #1\ : \ #2 \right\}}

\newcommand{\defn}[0]{:=}

\newcommand{\mc}{\mathcal}
\newcommand{\mb}{\mathbb}

\newcommand{\mr}{\mathrm}

\renewcommand{\Re}{\mathbb{R}}

\newcommand{\st}{\mr{s.t.}}

\usepackage[appendix=append]{apxproof}

\usepackage{setspace}
\title{Subgradient Methods for Nonsmooth Convex Functions with Adversarial Errors}

\author[1]{Martijn G\"osgens  \thanks{martijn.gosgens@cwi.nl}}
\author[1]{Bart P.G.\ Van Parys \thanks{bart.van.parys@cwi.nl}}

\affil[1]{CWI Amsterdam}

\begin{document}
\maketitle
\nocite{boyd2004convex}

\begin{abstract}
We consider minimizing nonsmooth convex functions with bounded subgradients.
However, instead of directly observing a subgradient at every step $k\in [0, \dots, N-1]$, we assume that the optimizer receives an \emph{adversarially corrupted} subgradient.
The adversary's power is limited to a finite corruption budget, but allows the adversary to strategically time its perturbations.
We show that the classical averaged subgradient descent method, which is optimal in the noiseless case, has worst-case performance that deteriorates \emph{quadratically} with the corruption budget. 
Using performance optimization programming, (i) we construct and analyze the performance of three novel subgradient descent methods, and (ii) propose a novel lower bound on the worst-case suboptimality gap of any first-order method satisfying a mild cone condition proposed by \cite{fatkhullin2025sgdhandleheavytailednoise}.
The worst-case performance of each of our methods degrades only \emph{linearly} with the corruption budget.
Furthermore, we show that the relative difference between their worst-case suboptimality gap and our lower bound decays as $\mc O(\log(N)/N)$, so that all three proposed subgradient descent methods are \emph{near-optimal}.
Our methods achieve such near-optimal performance without a need for momentum or averaging. This suggests that these techniques are not necessary in this context, which is in line with recent results by \cite{zamani2023exact}.
\end{abstract}

\section{Introduction}

We consider the classical problem of minimizing a nonsmooth convex objective function $f:X\to\Re$ with subgradients bounded in norm by $L$, i.e.,
\(
g\in \partial f(x) \implies \|g\|^2 \leq L^2
\)
for all $x$ in $X$, a linear vector space, and $g$ in its dual $X^\star$.
We assume that the problem is well-posed, i.e., $\min_{x}f(x) = f(x_\star) = f_\star>-\infty$ and $\| x_\star - x_0 \|^2 \leq R^2$ for $x_0\in X$. We will denote the problem class collecting all such functions as $\mc F$.

Subgradient methods are particularly simple iterative algorithms which have been studied following the pioneering work of \citet{shor1962application} with desirable properties for solving this class of optimization problems.
Starting from the initial iterate $x_0$, subgradient methods construct the sequence
\begin{align}
  \label{eq:subgradient-method}
  \begin{split}
    x_{k+1} = &~ x_{k} - h_{k} \tilde g_{k} \hspace{7em} \forall k\in [0, \dots, N-1],\\
    x_{N+1} = &~ \frac{\textstyle\sum_{k=0}^{N} h_k x_k}{\textstyle\sum^N_{k=0} h_k}
  \end{split}
\end{align}
with subgradients $\tilde g_k\in \partial f(x_k)$ for a fixed step size schedule $h = (h_k)_{k\in [0, \dots, N]}$.
Under the aforementioned assumptions, the performance of a generic fixed-step subgradient method satisfies for any $f\in \mc F$ the classical guarantee
\begin{equation}
  \label{eq:classical-upper-bound}
  f_{N+1}-f_\star \leq E(h) \defn \frac{R^2 + L^2 \textstyle\sum_{k=0}^N h^2_k}{2 \textstyle\sum_{k=0}^N h_k},
\end{equation}
where we write $f_k \defn f(x_k)$ for all $k\in [0, \dots, N+1]$; see for instance \citet{boyd2003subgradient, lan2020first}.

Remarkably, the {performance estimate} $E$ from \eqref{eq:classical-upper-bound} is a convex function of the step size schedule.
This simple observation enables designing an {optimized} subgradient method by considering the performance optimization problem
\(
h^\star\in\arg\min_{h\geq 0} E(h).
\)
This performance optimization problem admits the analytical solution
\begin{equation}
  \label{eq:classical-stepsize}
  E(h^\star) = \frac{RL}{\sqrt{N+1}} \quad {\rm{with}} \quad h^\star_k=\frac{R}{L\sqrt{N+1}}
\end{equation}
for all $k\in [0,\dots, N]$, i.e., fixed-step subgradient descent with subsequent iterate averaging.

We remark, however, that from the previous it does not immediately follow that this subgradient method is optimal since the classical performance estimate \eqref{eq:classical-upper-bound} is not particularly sharp.
Indeed, consider a vanishing step size schedule where $h_k=0$ for all $k\in [0, \dots, N-1]$. Then trivially we have $f_{N+1} - f_\star=f_0 -f_\star\leq R L$, whereas the performance estimate is degenerate for $R>0$.
More surprisingly,  \cite{zamani2023exact} have recently shown that a non-constant step size schedule
\begin{equation}
  \label{eq:zamani-stepsize}
  h_k = \frac{R(N-k)}{L(N+1)^{3/2}}
\end{equation}
for all $k\in [0, N-1]$ and $h_N=\infty$ (so that $x_{N+1}=x_N$) in fact enjoys the same performance guarantee $f_{N+1}-f_\star\leq E(h^\star)= \tfrac{RL}{\sqrt{N+1}}$ even though the classical performance estimate for this step size schedule is also degenerate.
Alternatively, performance estimation programming initiated by \citet{drori_performance_2014} allows to exactly characterize the worst-case performance of a generic subgradient as a tractable semidefinite optimization problem.
Although performance estimation programming has witnessed a surge of recent interest \citep{taylor_smooth_2017, das_gupta_branch-and-bound_2024}, finding a subgradient method with best worst-case performance results in a nonconvex performance optimization problem. In fact, verifying that a given subgradient method with step size schedule $h^\star$ enjoys the best worst-case performance is algorithmically hard.
Instead, subgradient methods with equal step sizes are shown to be worst-case optimal indirectly, by showing that no black-box optimization method can guarantee better performance \citep{Drori_Teboulle_2016}.

\subsection{Contributions}

In this paper, we generalize these results to a setting with adversarially corrupted subgradients. That is, the optimizer does not directly observe subgradients $g_k\in\partial f(x_k)$, but instead receives corrupted subgradients
\[
  \tilde g_k = g_k+e_k.
\]
We consider a bounded \emph{corruption budget} $\sum_{k=0}^{N-1}\|e_k\|^2 \leq \gamma^2$.
This means that the adversary must time its perturbations strategically. Such adversaries are of fundamental interest and have received a surge of recent attention in the optimization \citep{chang_gradient_2022}, bandit learning \citep{lykouris2018stochastic}, and adversarial neural networks \citep{wang2021robust} communities.

Clearly, if $\gamma=0$, then the problem studied here reduces to classical nonsmooth optimization admitting algorithms which reduce the suboptimality gap at rate $\mc O(N^{-1/2})$. On the other hand, if $\gamma\ge L\sqrt{N}$, then the adversary can fully corrupt the subgradients by the choice $e_k=-g_k$, so that $x_{N+1}=x_0$ and no progress can be made.
Because of this, we study the interesting intermediate regime $\gamma\in(0, L\sqrt{N})$.

The following result illustrates that the classical subgradient descent method may suffer arbitrarily poor performance in the presence of adversarial noise:
\begin{lemma}\label{lem:example}
    Let $\gamma\in(0,L\sqrt N)$.
    For the classical subgradient method with step sizes given in Equation (\ref{eq:classical-stepsize}), there exists a problem instance where
    \[
    f_{N+1}-f_\star\ge\frac{\gamma^2R}{8L\sqrt{N+1}}
    \]
    which exceeds the trivial bound $RL$ for $\gamma>2\sqrt2(N+1)^{1/4}L$ and grows unbounded for $\gamma\gg N^{1/4}$.
\end{lemma}
\begin{appendixproof}[Proof of~\cref{lem:example}.]
    Consider the resisting function
\[
f(x)=\frac{\gamma}{2\sqrt{N}}|x|.
\]
Note that this function is convex, has minimizer $x_\star=0$ and is $L$-Lipschitz, since
\[
\|g_k\|\le\frac{\gamma}{2\sqrt{N}}\le \frac{L}{2}.
\]
We let the adversary corrupt the subgradient with noise $e_k=-g_k-\frac{\gamma}{2\sqrt{N}}$, which is within the budget, since
\[
\sum_{k=0}^{N-1}\|e_k\|^2\le\sum_{k=0}^{N-1}(\gamma/2\sqrt{N}+\gamma/2\sqrt{N})^2=\gamma^2. 
\]
We pick $x_0=0$, so that the iterates become
\[
x_k=x_{k-1}-h_{k-1}\tilde g_k=x_{k-1}+h\frac{\gamma}{2\sqrt N}=\frac{\gamma}{2\sqrt{N}}kh.
\]
The averaging step results in
\[
x_{N+1}=\frac{\gamma h}{2(N+1)\sqrt{N}}\sum_{k=0}^{N}k=\frac{\gamma h\sqrt{N}}{4}.
\]
So that the suboptimality gap is 
\[
f_{N+1}-f_\star=\frac{\gamma}{2\sqrt{N}}\cdot \frac{\gamma h\sqrt{N}}{4}=\frac{\gamma^2 h}{8},
\]
which exceeds the trivial bound $RL$ for $\gamma^2\ge 8RL/h$. 
The result follows after substituting $h=h^*$ from~\eqref{eq:classical-stepsize}.
\end{appendixproof}
\cref{lem:example} shows that any performance guarantee for the classical subgradient descent grows at least \emph{quadratically} with $\gamma$. However, we will show in this work that it is possible to construct a step size schedule with a performance guarantee that only grows \emph{linearly} with $\gamma$, for $\gamma\in(0,L\sqrt{N})$.

We derive a performance estimate which yields a convex performance optimization program resulting in a (nearly) optimal subgradient method. Our performance guarantee (which we describe in~\cref{cor:u-S-bound}) is of the form
\begin{equation}\label{eq:s-bound-main}
f_{N+1}-f_\star\le \frac{RL}{\sqrt{N+1}}u_N^{\mathbb S}(\gamma/L),
\end{equation}
where $u_N^{\mathbb S}(\sigma)\ge1$ can be computed by solving a convex semidefinite optimization problem.
Comparing~\eqref{eq:s-bound-main} to the classical bound~\eqref{eq:classical-stepsize},
we see that the perturbations affect the suboptimality gap by a factor that only depends on $N$ and $\sigma\defn\gamma/L$.
To prove that~\eqref{eq:s-bound-main} is close to optimal, we construct an auxiliary dual performance optimization problem with an almost matching performance lower bound.
The performance lower bound holds for a class of optimization methods which satisfy the cone condition $x_0-x_k\in\text{cone}(\tilde g_0,\dots,\tilde g_{k-1})$.
This class contains any subgradient method with non-negative step sizes, the Nesterov accelerated gradient descent method and indeed most practically relevant variable step size algorithms; see also \cite{fatkhullin2025sgdhandleheavytailednoise}.
For any such optimization method, we prove in~\cref{cor:lower} that there exists a problem instance where
\[
f_{N+1}-f_\star\ge\frac{RL}{\sqrt{N+1}}\ell_N(\sigma).
\] 
For $\gamma=0$ (i.e., uncorrupted subgradients), our performance bound coincides with the known universal lower bound for (uncorrupted) nonsmooth optimization~\citep{Drori_Teboulle_2016}. 

To obtain analytic performance guarantees, we further bound our convex performance optimization program. In~\cref{lem:admissible}, this leads to explicit formulas for a step size schedule with performance guarantee
\begin{equation}\label{eq:admissible-bound-main}
f_{N+1}-f_\star\le \frac{RL}{\sqrt{N+1}}u(\sigma),
\end{equation}
where the function $u(\sigma)\ge1$ is defined implicitly as the solution in
\[
\sigma^2=u^2-1-2\log u.
\]
We thus see that this factor only depends on $\sigma$.
The fact that the performance guarantee~\eqref{eq:admissible-bound-main} grows linearly with $\sigma$ (and hence also $\gamma$) follows from the bound $u(\sigma)\le 1+\sigma$.
The bound~\eqref{eq:admissible-bound-main} implies that we need $N=\bigO{\left((\gamma+L)\frac R\varepsilon\right)^2}$
iterations to achieve $f_{N+1}-f_\star\le \varepsilon$.
We further prove in~\cref{thm:squeeze} that for all $N\ge1$ and $\sigma\in[0,\sqrt{N}]$, it holds that
\[
\left(1-\frac52\frac{\log (N+1)}{N}\right) u(\sigma)\le \ell_N(\sigma)\le u_N^{\mathbb S}(\sigma)\le u(\sigma).
\]
This means that the performance guarantees of both the subgradient method associated with the convex semidefinite performance optimization problem and the explicit subgradient method attain a worst-case suboptimality gap which is asymptotically equivalent to the universal lower bound.

The explicit step sizes that achieve the performance guarantee~\eqref{eq:admissible-bound-main} are of the form
\begin{equation}\label{eq:admissible-steps-main}
h_k=\frac{R(N-k)}{L(N+1)^{3/2}}\cdot \frac{u(\sigma)}{u(\sigma)^2-(u(\sigma)^2-1)\frac{k}{N+1}}\cdot \xi_N(\sigma),
\end{equation}
for $k \in [0,\dots,N-1]$ and $h_N=\infty$. Here, $\xi_N(\sigma)$ is a small correction factor that we describe in~\cref{sec:analysis}. This correction factor satisfies the bounds $1\le \xi_N(\sigma)\le1+\frac2N$ for all $N\ge1$ and $\sigma\in[0,\sqrt N]$.
Moreover, the method given in~\eqref{eq:admissible-steps-main} does not use averaging or momentum, which suggests that these techniques are not necessary in this setting.

\subsection{Related Work}
The problem of convex optimization with \emph{exact} subgradients has been studied extensively. In \emph{smooth} optimization, the function $f$ is assumed to have Lipschitz continuous gradients. \citet{nesterov1983method} proposed \emph{Fast Gradient Descent} (FGM), which outperforms the classical gradient descent by an order of magnitude via a \emph{momentum} technique. 
This momentum technique has been further refined by \emph{Optimized Gradient Descent} (OGM,~\citet{kim_optimized_2016}), which improves the performance guarantee by a constant factor.
We refer to~\cite{nesterov_lectures_2018} for a complete overview of convex optimization with exact (sub)gradients. 

In many practical problems, it is infeasible or even impossible to obtain exact (sub)gradients.
\citet{liu2024nonasymptotic} give several examples of applications where exact gradients are unavailable. For example, when gradients need to be approximated by finite difference formulas or when the evaluating $f$ involves solving another optimization problem~\citep{ghadimi2018approximation}.
Perhaps the most common application where computing exact gradients is infeasible comes from training machine learning models on large data sets: computing an exact gradient of the loss function requires an iteration over the entire training set, which can be prohibitively expensive. To overcome this, one can randomly sample from the training set to obtain an unbiased estimate of the gradient~\citep{bottou2010large}.
This results in \emph{Stochastic Gradient Descent} (SGD), where the gradient perturbations are modeled by random variables~\citep{robbins1951stochastic,kiefer1952stochastic}. In the SGD literature, these perturbations are typically assumed to be unbiased and independent.

In other optimization problems, however, it may not be realistic to assume that the perturbations are unbiased and independent. In those settings, it makes sense to pose deterministic constraints on the perturbations and consider worst-case performance. Optimization with inexact gradients has mainly been studied in the context of smooth optimization:

\citet{devolder2014first} consider smooth optimization in a setting where both the gradient and the function value are inexact. They assume that the observed function value and gradient satisfy stage-wise constraints.
It is observed that momentum methods are more vulnerable to error accumulation than standard gradient descent methods.
\citet{liu2024nonasymptotic} use PEP to derive performance bounds of OGM and FGM for smooth optimization and stage-wise bounded errors, i.e., $\|e_k\|\le \varepsilon$ for every $k$. Their results confirm that these momentum methods are sensitive to accumulation of errors. The considered stage-wise corruption constraints mean that the adversary does not have to time their corruptions strategically, in contrast to the total corruption budget constraint that we consider in this work.

Instead of these stage-wise corruption budgets, \citet{chang_gradient_2022} limit the adversarial power by constraining the cumulative corruption $\sum_{i=0}^k\|e_i\|$ for every $k\in [0, \dots, N-1]$. They assume the objective function $f$ satisfies the Polyak-\L{}ojasiewicz smoothness condition and derive performance guarantees for gradient descent methods with variable step sizes.
In contrast to their step-wise cumulative corruption budgets, we consider a single total corruption budget, which allows the adversary to corrupt the first subgradients more heavily.

\citet{schmidt_convergence_2011} study \emph{proximal gradient descent} methods where both the gradient and the proximal operator are inexact. For constant step sizes, they prove a performance guarantee that increases quadratically with the total error $\sum_{k=0}^{N-1}\|e_k\|$.
\citet[Theorem 2]{atchade2017perturbed} extends these results to non-constant step sizes and provides sufficient conditions on the step size sequence and the perturbation sequence $\|e_k\|$ to guarantee convergence of the optimization method.

\citet{alistarh2018byzantine} combine SGD with adversarial corruptions in a distributed setting where a fraction of the `workers' provide adversarial gradient information. They show how SGD can be adapted to be robust to these \emph{Byzantine} failures.
Similarly, \citet{wang2021robust} study SGD for training neural networks in a setting where a fraction of the training data has been adversarially corrupted.
In the context of stochastic bandits, there have been similar efforts to robustify algorithms to adversarially corrupted output~\citep{lykouris2018stochastic}.





\subsection*{Notation}
For sequences $a_N,b_N$, we write $a_N\ll b_N$ or $a_N=o(b_N)$ if $\lim_{N\to\infty}\frac{a_N}{b_N}=0$. We write $a_N\sim b_N$ if $\lim_{N\to\infty}\frac{a_N}{b_N}=1$, and we write $a_N=\bigO{b_N}$ if there exists a $c>0$ such that $|a_N|\le c|b_N|$ for all $N$. 

\section{Performance Optimization Problems}
In this section, we provide optimization programs that yield upper and lower bounds on the worst-case suboptimality gap.

\subsection{Admissible Subgradient Methods}
Given that in the noiseless case, a simple subgradient algorithm with fixed step size schedule is optimal, it is natural to also consider these methods in the context of adversarial gradient noise. 
In fact, in what follows we will restrict attention to \emph{admissible} subgradient methods ($h\in \mc H$) for which
\(
\sum_{j=k+1}^N \tfrac{h_j}{(N-k)}
\)
is nondecreasing in $k\in [0, \dots, N-1]$. The main contribution of this section is to  construct a desirable performance estimate for such admissible subgradient methods. In Section \ref{sec:univ-lower-bound} we will quantify the extent to which this restriction causes a loss of optimality.

\begin{toappendix}
  \begin{lemma}[Admissible Subgradient Methods]
    We have that $h\in \mc H$ if and only if there exists $\lambda'_0\in \Re_+, \dots, \lambda'_{N+1}\in \Re_+$ satisfying
    \begin{equation}\label{eq:lambda-condition}
      (N-k) \lambda'_k - \sum_{i=0}^{k-1} \lambda'_i = \frac{1}{(N+1)} - \frac{h_k}{\sum_{i=0}^N h_i} \quad \forall k\in [0, \dots, N-1].
    \end{equation}   
  \end{lemma}
  \begin{proof}
    Fix step sizes $h_0,\dots,h_{N}\ge0$ and pick $\lambda'_0,\dots,\lambda_{N+1}'$ to satisfy~\eqref{eq:lambda-condition} (possibly with $\lambda_k<0$ for some values). We will prove that $\lambda_k\ge0$ for all $k$ is equivalent to $\sum_{j=k+1}^N \tfrac{h_j}{(N-k)}$ being nondecreasing.
    Summing~\eqref{eq:lambda-condition} for $k=0,\dots,j$ yields
    \[
    (N-j)\sum_{i=0}^j\lambda'_i=\frac{j+1}{N+1}-\frac{\sum_{k=0}^j h_k}{\sum_{i=0}^Nh_i}=\frac{\sum_{k=j+1}^N h_k}{\sum_{i=0}^Nh_i}-\frac{N-j}{N+1},
    \]
    which can be rewritten to
    \[
    \left(\sum_{i=0}^Nh_i\right)\sum_{i=0}^j\lambda'_i=\frac{\sum_{k=j+1}^N h_k}{N-j}-\frac{\sum_{i=0}^Nh_i}{N+1}.
    \]
    From this equation, we can see that $\sum_{i=0}^j\lambda'_i$ is nondecreasing if and only if $\frac1{N-j}\sum_{k=j+1}^N h_k$ is nondecreasing. Finally, the $\lambda_k'$ are all nonnegative if and only if the sum $\sum_{i=0}^j\lambda'_i$ is nondecreasing and $\lambda'_0\ge0.$ Thus, we still need to verify $\lambda'_0\ge0.$ For $k=0$, Equation~\eqref{eq:lambda-condition} reads
    \[
    N\lambda'_0=\frac1{N+1}-\frac{h_0}{\sum_{i=0}^N h_i},
    \]
    which is nonnegative if and only if
    \[
    h_0\le \frac1{N}\sum_{i=1}^N h_i.
    \]
    This is equivalent to
    \[
    \frac{1}{N+1}\sum_{i=0}^N h_i=\frac1{N+1}h_0+\frac{N}{N+1}\left(\frac1N\sum_{i=1}^N h_i\right)\le \left(\frac1{N+1}+\frac{N}{N+1}\right)\cdot\frac1N\sum_{i=1}^N h_i,
    \]
    which completes the proof.
  \end{proof}
\end{toappendix}

\begin{algorithm}[!ht]
\small
\DontPrintSemicolon
\KwIn{Function $f:\Re^d\to\Re$, number of iterations $N$, step size schedule $h \in \mc H$ and initial iterate $x_0\in\Re^d$.}
\BlankLine
\For{ $k=0,\dots,N-1$}{
  \BlankLine
Retrieve a noisy subgradient $\tilde g_k \in \partial f(x_k)+e_k$.
\BlankLine
$x_{k+1}=x_k-h_k \tilde g_k$
}
\BlankLine
\KwOut{
$x_{N+1} = \tfrac{\textstyle\sum_{k=0}^{N} h_k x_k}{\textstyle\sum^N_{k=0} h_k}$.
} 
\caption{Admissible subgradient method with step sizes $h_k$.}
\label{alg:step-sizes}
\end{algorithm}

A characteristic property of any fixed step size subgradient method is that we may write
\begin{equation}
  \label{eq:fsfom}
  x_{N+1} = x_0 - \textstyle \sum_{k=0}^{N-1} \alpha_{k} \tilde g_k.
\end{equation}
In other words, the final iterate $x_{N+1}$ is equal to the initial iterate and a conic combination of the noisy subgradients observed along the way. 
Straightforward calculation indicates that the relevant conic combination can be deduced from the step size schedule as
\(
  \alpha_k = h_k \cdot \textstyle\tfrac{\sum_{i=k+1}^N h_i}{\sum_{i=0}^N h_i} \geq 0
\)
for all $k\in [0, \dots, N-1]$. It is noteworthy to point out that two distinct subgradient methods for different step size schedules can be associated with the same conic combination $\alpha$. We call two subgradient methods equivalent if they share the same conic combination $\alpha$.  We denote by $$\mc H(\alpha) = \set{h\in \mc H}{h_k \cdot \textstyle\tfrac{\sum_{i=k+1}^N h_i}{\sum_{i=0}^N h_i}=\alpha_k ~~\forall k\in [0, \dots, N-1]}$$ the equivalence class of admissible subgradient methods associated with a particular conic combination $\alpha$.
A quick calculation reveals that both the classical subgradient method with averaging (see Equation (\ref{eq:classical-stepsize})) as well as the subgradient method of \cite{zamani2023exact} (see Equation (\ref{eq:zamani-stepsize})) satisfy Equation (\ref{eq:fsfom}) for 
\begin{equation}\label{eq:alpha-noiseless}
\alpha_k^* = \frac{R(N-k)}{L\sqrt{N+1}^3},    
\end{equation}
and are therefore equivalent.
Moreover, for any $h_N\in [\tfrac{R}{(L\sqrt{N+1})},\infty]$, we can find $h_0,\dots,h_{N-1}$ so that $h\in \mc H(\alpha^\star)$. We reveal in this section that this implies that there is a manifold of optimal subgradient methods which interpolates between the classical subgradient method with averaging $(h_N=\tfrac{R}{(L\sqrt{N+1})})$ and the subgradient method from \cite{zamani2023exact} (corresponding to $h_N=\infty$).

In the following theorem we will advance a performance estimate which depends on the step size schedule only through its associated conic combination $\alpha$.
That is, two equivalent subgradient methods will enjoy precisely the same performance guarantee. 
In particular, this suggests to design a subgradient method by optimizing over the conic combination $\alpha$ rather than the step size schedule directly.

\begin{toappendix} 
\begin{lemma}\label{lem:bound}
  Consider a step size schedule $h\in \mc H(\alpha)$ satisfying the condition give in Equation (\ref{eq:lambda-condition}).
    Then, the suboptimality gap of the averaged iterate $x_{N+1}$ satisfies the bound
    \begin{equation}
        f_{N+1}-f_*\le \sum_{k=0}^N \frac{1}{N+1} \langle g_k , x_0 - x_\star \rangle - \sum_{k=0}^N \sum_{j=0}^{k-1} \frac{\alpha_j}{N-j} \langle g_k, g_j+e_j\rangle.
    \end{equation}
\end{lemma}
\begin{proof}[Proof of Lemma \ref{lem:bound}]
    By convexity, the subgradients satisfy the following inequalities
\begin{align*}
  f_\star & \geq f_k + \langle g_k , x_\star - x_k \rangle \quad \forall k\in [0, \dots, N]\\
  f_k &\geq f_{N+1} + \langle g_{N+1} , x_k - x_{N+1} \rangle \quad \forall k \in [0, \dots, N]\\
  f_i & \geq f_k + \langle g_k, x_i-x_k\rangle \quad \forall k \in [0, \dots, N],~ \forall i\in [0, N]: ~k\geq i+1.
\end{align*}
Summing the constraints after multiplying by $1/(N+1)$, $h_k/\sum_{i=0}^N h_i$ and $\lambda_i\geq 0$, respectively, gives
\begin{align*}
&f_*+\frac1{\sum_{k=0}^Nh_k}\sum_{k=0}^Nh_kf_k+\sum_{i=0}^N\sum_{k=i+1}^N\lambda_if_i \ge\\
 &\frac1{N+1}\sum_{k=0}^N\left( f_k+\langle g_k , x_\star - x_k \rangle\right) +f_{N+1}+\left\langle g_{N+1},\frac1{\sum_{k=0}^Nh_k}\sum_{k=0}^Nh_kx_k-x_{N+1}\right\rangle+\sum_{i=0}^N\sum_{k=i+1}^N\lambda_i\left(f_k+\langle g_k,x_i-x_k\rangle\right).
\end{align*}
Our choice of $x_{N+1}$ yields
\[
\frac1{\sum_{k=0}^Nh_k}\sum_{k=0}^Nh_kx_k-x_{N+1}=0,
\]
so that one of the inner products vanishes.
Interchanging the double sum and re-ordering terms yields
\begin{align*}
  & \sum_{k=0}^N \frac{1}{N+1} \langle g_k , x_k - x_\star \rangle + \sum_{k=0}^N \sum_{i=0}^{k-1} \lambda_i \langle g_k, x_k-x_i\rangle \\
  \geq & f_{N+1} - f_\star +  \sum_{k=0}^N f_k\left(\frac{1}{N+1} -\left(\frac{h_k}{\sum_{i=0}^N h_i}\right) - (N-k)\lambda_k + \sum_{i=0}^{k-1} \lambda_i  \right).
\end{align*}
Take now $\lambda_i\ge0$ satisfying Equation \eqref{eq:lambda-condition} to get a valid bound.
Summing Equation \eqref{eq:lambda-condition} over $k=0,\dots, j$  gives the equivalent set of conditions
\[
  (N-j) \sum_{i=0}^{j} \lambda_i = \tfrac{(j+1)}{(N+1)} - \tfrac{\sum_{k=0}^{j}h_k}{\sum_{i=0}^N h_i} \quad \forall j\in [0, \dots, N-1].
\]
Or equivalently,
\begin{equation}\label{eq:lambdas-summed}
\sum_{i=0}^{j} \lambda_i = \frac{1}{N-j}\left(1-\frac{N-j}{(N+1)} - \tfrac{\sum_{i=0}^{j}h_i}{\sum_{i=0}^N h_i}\right)=\frac{\sum_{i=j+1}^{N}h_i}{(N-j)\sum_{i=0}^N h_i}-\frac1{N+1}.
\end{equation}
With this choice of $\lambda_i$ our bound becomes
\begin{equation}\label{eq:bound1}
f_{N+1}-f_*\le \sum_{k=0}^N \frac{1}{N+1} \langle g_k , x_k - x_\star \rangle + \sum_{k=0}^N \sum_{i=0}^{k-1} \lambda_i \langle g_k, x_k-x_i\rangle.
\end{equation}
We use
\begin{equation}
  \label{eq:linearity-iterations}
  x_k-x_i=-\sum_{j=i}^{k-1}h_j (g_j+e_j)
\end{equation}
to rewrite
\begin{align}
\sum_{i=0}^{k-1} \lambda_i \langle g_k, x_k-x_i\rangle
&=-\sum_{i=0}^{k-1} \sum_{j=i}^{k-1} \lambda_i h_j \langle g_k, g_j+e_j\rangle=-\sum_{j=0}^{k-1}h_j\langle g_k, g_j+e_j\rangle\sum_{i=0}^j\lambda_i
\label{eq:gradient-inner-product}
\end{align}
Multiplying Equation~\eqref{eq:lambdas-summed} with $h_j$ yields
\[
h_j\sum_{i=0}^{j} \lambda_i =\frac{\alpha_j}{N-j}-\frac{h_j}{N+1}.
\]
Substituting this into~\eqref{eq:gradient-inner-product} and~\eqref{eq:bound1} improves our bound to
\[
f_{N+1}-f_*\le \sum_{k=0}^N \frac{1}{N+1} \langle g_k , x_k - x_\star \rangle - \sum_{k=0}^N \sum_{j=0}^{k-1} \left(\frac{\alpha_j}{N-j}-\frac{h_j}{N+1}\right) \langle g_k, g_j+e_j\rangle.
\]
Similarly, substituting
\[
x_k=x_0-\sum_{j=0}^{k-1}h_j(g_j+e_j)
\]
yields
\[
f_{N+1}-f_*\le \sum_{k=0}^N \frac{1}{N+1} \langle g_k , x_0 - x_\star \rangle - \sum_{k=0}^N \sum_{j=0}^{k-1} \frac{\alpha_j}{N-j} \langle g_k, g_j+e_j\rangle,
\]
which completes the proof.
\end{proof}
\end{toappendix}

\begin{proposition}
  \label{thm:main-upper-bound}
  Consider Algorithm~\ref{alg:step-sizes} with step size schedule $h\in \mc H(\alpha)$.
  For any noise level $\gamma \geq 0$, Algorithm \ref{alg:step-sizes} satisfies
  \begin{align*}
    & f_{N+1} - f_\star \leq E(\alpha)\defn \\
    &     \begin{array}{r@{~}l}
            RL\cdot \min & \left(\tau \sigma^2 +  \nu_\star  + \sum_{k=0}^{N} \nu_k\right) \\[1em]
             \st & \tau\in \Re_+, ~\nu_k\geq 0 ~~\forall k \in [\star, 0, \dots, N], \\[1em]
         & 0\preceq \Lambda(\nu, \tau, \alpha L/R) = \\
         & 
           \left(\begin{array}{c@{\,~}c@{\,~}c@{\,~}c@{\,~}c@{\,~}c@{\,~}c|c@{\,~}c@{\,~}c@{\,~}c@{\,~}c}
        \nu_\star & \frac{-1}{2(N+1)} & \frac{-1}{2(N+1)} &\frac{-1}{2(N+1)}  & \dots & \frac{-1}{2(N+1)} & \frac{-1}{2(N+1)} & 0 & 0 & 0 & \dots & 0  \\
        \frac{-1}{2(N+1)} &\nu_0 & \frac{\alpha_0 L}{2N R} & \frac{\alpha_0L}{2NR} & \dots & \frac{\alpha_0L}{2NR} & \frac{\alpha_0L}{2NR} & 0 & 0 & 0& \dots & 0  \\
        \frac{-1}{2(N+1)} & \frac{\alpha_0L}{2NR} & \nu_1 & \frac{\alpha_1L}{2(N-1)R} & \dots & \frac{\alpha_1L}{2(N-1)R}  & \frac{\alpha_1L}{2(N-1)R} & \frac{\alpha_0L}{2NR} & 0 & 0 & \dots &0   \\
        \frac{-1}{2(N+1)} & \frac{\alpha_0L}{2NR} & \frac{\alpha_1L}{2(N-1)R} & \nu_2 &\dots  & \frac{\alpha_2L}{2(N-2)R} &  \frac{\alpha_{2}L}{2(N-2)R} & \frac{\alpha_0L}{2NR} & \frac{\alpha_1L}{2(N-1)R}& 0 & \dots & 0 \\
        \vdots & \vdots & \vdots & \vdots &  \ddots & \vdots & \vdots &\vdots &\vdots & \vdots & \ddots & \vdots \\
        \frac{-1}{2(N+1)} & \frac{\alpha_0L}{2NR} & \frac{\alpha_1L}{2(N-1)R} & \frac{\alpha_2L}{2(N-2)R} & \dots& \nu_{N-1} & \frac{\alpha_{N-1}L}{2R} & \frac{\alpha_0L}{2NR} & \frac{\alpha_1L}{2(N-1)R} & \frac{\alpha_2L}{2(N-2)R} & \dots & 0 \\
        \frac{-1}{2(N+1)} & \frac{\alpha_0L}{2NR} & \frac{\alpha_1L}{2(N-1)R} & \frac{\alpha_2L}{2(N-2)R} & \dots & \frac{\alpha_{N-1}L}{2R} & \nu_N & \frac{\alpha_0L}{2NR} & \frac{\alpha_1L}{2(N-1)R} & \frac{\alpha_2L}{2(N-2)R} & \dots &  \frac{\alpha_{N-1}L}{2R} \\
        \hline
        0 & 0 & \frac{\alpha_0L}{2NR} & \frac{\alpha_0L}{2NR} & \dots & \frac{\alpha_0L}{2NR} & \frac{\alpha_0L}{2NR}  & \tau & 0 & 0 & \dots & 0 \\
        0 & 0 & 0 & \frac{\alpha_1L}{2(N-1)R}& \dots & \frac{\alpha_1L}{2(N-1)R} & \frac{\alpha_1L}{2(N-1)R}  & 0 & \tau & 0 & \dots & 0 \\
        0 & 0 & 0 & 0 & \dots & \frac{\alpha_2L}{2(N-2)R} & \frac{\alpha_2L}{2(N-2)R}  & 0 & 0 & \tau & \dots & 0 \\
        \vdots & \vdots & \vdots & \vdots & \ddots & \vdots & \vdots  & \vdots & \vdots & \vdots & \ddots & \vdots \\
      0 & 0 & 0 & 0 & \dots & 0 & \frac{\alpha_{N-1}L}{2R}  & 0 & 0 & 0 & \dots & \tau \\
    \end{array}\right).
  \end{array}
  \end{align*}
\end{proposition}
\begin{appendixproof}[Proof of Theorem \ref{thm:main-upper-bound}.]
    The right-hand side in the bound from Lemma \ref{lem:bound} can be rewritten to
    \begin{align*}
      \sum_{k=0}^N \frac{1}{N+1} \langle g_k , x_0 - x_\star \rangle - \sum_{k=0}^N \sum_{j=0}^{k-1} \frac{\alpha_j}{N-j} \langle g_k, g_j+e_j\rangle = &  \langle -\Lambda(0, 0, \alpha), G\rangle
    \end{align*}
    where the Grammian $G = [x_0-x_\star| g_0, \dots, g_{N}| e_0,\dots, e_{N-1}]^\top [x_0-x^\star| g_0, \dots, g_{N}| e_0,\dots, e_{N-1}]\rangle\succeq 0$ collects all relevant inner products. Hence,
    \begin{align*}
      & f_{N+1} - f_\star\\
      \leq & \left\{
             \begin{array}{rl@{\hspace{4em}}l}
               \max & \multicolumn{2}{l}{\langle -\Lambda(0, 0, \alpha), G \rangle} \\[0.5em]
               \st & G(x_0-x_\star, x_0-x_\star) = \norm{x_0-x_\star}^2 \leq R^2 & [{\rm{Dual~Variable:}}~ \nu_\star\geq 0]\\[0.5em]
                    & G(g_k, g_k) = \norm{g_{k}}^2\leq L^2 & [{\rm{Dual~Variable:}}~ \nu_k\geq 0] ~~\forall k\in [0, \dots, N]\\[0.5em]
                    & \sum_{k=0}^{N-1}  G(e_k, e_k) = \sum_{k=0}^{N-1} \norm{e_{k}}^2\leq \gamma^2& [{\rm{Dual~Variable:}}~ \tau\geq 0].
             \end{array}\right.
    \end{align*}
    The claimed result now follows from a standard application Lagrangian duality. 
  \end{appendixproof}
Since $h=(\alpha_0, \dots, \alpha_{N-1}, \infty) \in \mc H(\alpha)$, we have that $\mc H(\alpha)\neq \emptyset \iff \alpha\geq 0$. Hence, the performance optimization problem of finding the subgradient method with best performance estimate reduces to the following convex semidefinite optimization problem:

\begin{corollary}\label{cor:u-S-bound}
    Let $u_N^{\mathbb S}(\sigma)$ be given by
    \begin{align}
  \label{eq:gen_pop}
  u^{\mathbb S}_N(\sigma) \defn &  \left\{\!\!
          \begin{array}{r@{~}l}
            \sqrt{N+1}\cdot \min & \tau \sigma^2 + \nu_\star + \sum_{k=0}^{N} \nu_k \\[1em]
            \st & \alpha\in \Re_+^{N}, ~\tau\in \Re_+, ~\nu_k\geq 0 \quad \forall k \in [\star, 0, \dots, N], \\[1em]
                 &  \Lambda(\nu, \tau, \alpha L/R) \succeq 0,
           \end{array}\right.
\end{align}
and let $\alpha^*$ denote the corresponding solution. Any $h\in\mc H(\alpha^*)$ enjoys the performance guarantee
\[
f_{N+1}-f_\star\le \frac{RL}{\sqrt{N+1}}u^{\mathbb S}_N(\sigma).
\]
\end{corollary}

Although the performance optimization problem (\ref{eq:gen_pop}) reduces to a semidefinite optimization problem, it is difficult to analyze analytically. In the following result, we introduce a more manageable second-order cone performance optimization problem which will allow us to construct near optimal analytic step size schedules in Section~\ref{sec:analysis}. 

\begin{proposition}\label{prop:u-L-bound}
  Consider the convex optimization problem
  \begin{equation}
    \label{eq:exact-convex-procedure}
    \begin{array}{rl}
      (u^{\mathbb L}_N(\sigma))^2 \defn \min  & \frac{\sigma^2 + \sum_{k=0}^{N} y_k }{(N+1)y_0}\\
                             \st & y_0\geq 0,\\
                                   & y_{k+1} \geq y_{k} + y_{k}^2 \quad \forall k\in [0, \dots, N-1].
    \end{array}
  \end{equation}
  Let
  \[
    \alpha^{\mathbb L}_k = \frac{R}{L}\cdot\frac{N-k}{(N+1)^{3/2}} \cdot \frac{y_k^\star }{y_0^\star u^{\mathbb L}_N(\sigma)} \quad \forall k\in [0, \dots, N-1]
  \]
  with $y^\star$ an optimal solution in \eqref{eq:exact-convex-procedure}.
  Any subgradient method with step size schedule $h\in \mc H(\alpha^{\mb L})$ enjoys the performance guarantee
  \[
  f_{N+1} - f_\star\leq \frac{RL}{\sqrt{N+1}}\cdot u^{\mathbb L}_N(\sigma).
  \]

\end{proposition}
\begin{appendixproof}[Proof of Proposition \ref{prop:u-L-bound}.]

Following the partition of $\Lambda$ into submatrices from Theorem \ref{thm:main-upper-bound} we write 
\[
\Lambda(\nu, \tau, \alpha) =\left(\begin{array}{cc}
    A & B \\
    B^\top & \tau\cdot I
\end{array}\right).
\]
By Schur's complement, $A-\frac1\tau B B^\top\succeq 0$ and $\tau\geq 0$ are sufficient to guarantee $\Lambda(\nu, \tau, \alpha)\succeq 0$.
We now calculate Schur's complement.
For $i\in[0,\dots,N]$ and $j\in[0,\dots,N-1]$, we write
\[
B_{ij}=\begin{cases}
    0&\text{ if }i\le j+1,\\
    \frac{\alpha_j}{2(N-j)}\text{ else.}
\end{cases}
\]
For $i,j\in[0,\dots,N]$, we calculate
\[
(B B^\top)_{ij}=\sum_{k=0}^{N-1}B_{ik}B_{jk}=\sum_{k=0}^{\min\{i,j\}-1}B_{ik}B_{jk}=\frac14\sum_{k=0}^{\min\{i,j\}-1}\frac{\alpha_k^2}{(N-k)^2}.
\]
Let us define
\[
\mu_\ell=\frac1{4\tau}\sum_{k=0}^{\ell-1}\frac{\alpha_k^2}{(N-k)^2},
\]
so that $\frac1\tau(B B^\top)_{ij}=\mu_{\min\{i,j\}-1}$. 
Then Schur's complement is given by
\[
\left(A-\frac1\tau BB^\top\right)_{ij}=\begin{cases}
    \nu_*&\text{ if }i=j=0,\\
    \frac{-1}{2(N+1)}&\text{ if }i=0\vee j=0,\\
    \nu_{i-1}-\mu_{i-1}&\text{ if }i=j>0,\\
    \frac{\alpha_{\min\{i,j\}-1}}{2(N+1-\min\{i,j\})}-\mu_{\min\{i,j\}-1}&\text{ else.}
\end{cases}
\]

We set the dual variable $\nu$ to
\begin{align*}
  \nu_\star = & \frac{N}{2(N+1)^2 \tilde \alpha_0}\\
  \nu_0 = & \frac{\alpha_0}{2N}\\
  \nu_{k+1} = & \nu_{k} + \frac{\alpha_{k}^2}{4 (N-k)^2 \tau} \quad \forall k \in [0, \dots, N-1].
\end{align*}
With these variables, the diagonal entries are 
\[
\left(A-\frac1\tau BB^\top\right)_{ii}=\begin{cases}
    \frac{N}{2(N+1)^2 \alpha_0}&\text{ if }i=0,\\
    \frac{\alpha_0}{2N}&\text{ else.}
\end{cases}
\]
Again applying Schur, we can write
\[
A-\frac1\tau B B^\top=\left(\begin{array}{cc}
    \nu_* & \frac{1}{2(N+1)}\cdot \mathbf{1}^\top \\
    \frac{1}{2(N+1)}\cdot \mathbf{1} & C
\end{array}\right),
\]
where 
\[
C_{ij}=\begin{cases}
    \frac{\alpha_0}{2N}&\text{ if }i=j,\\
    \frac{\alpha_{\min\{i,j\}}}{2(N-\min\{i,j\})}-\mu_{\min\{i,j\}}&\text{ else.}
\end{cases}
\]
Then $A-\frac1\tau BB^\top$ is positive semidefinite iff
\[
C-\frac{\alpha_0}{2N}\mathbf{1}\cdot\mathbf{1}^\top
\]
is positive semidefinite. This matrix has a zero diagonal. The off-diagonal elements are given by
\[
\frac{\alpha_{\min\{i,j\}}}{2(N-\min\{i,j\})}-\mu_{\min\{i,j\}}-\frac{\alpha_0}{2N}.
\]
Suppose that
\[
  \frac{\alpha_{k+1}}{N-k-1} = \frac{\alpha_{k}}{N-k} + \frac{\alpha^2_{k}}{2 (N-k)^2 \tau} \quad \forall k\in [0, \dots, N-2]
\]
then the off-diagonal elements are zero. Indeed, for $\min\{i,j\}=0$ we have trivially,
\[
    \frac{\alpha_{0}}{2N}-\mu_{0}-\frac{\alpha_0}{2N}=0.
\]
Observe that for any $k$ we have
\begin{align*}
& \frac{\alpha_{k+1}}{2(N-k-1)}-\mu_{k+1}-\frac{\alpha_0}{2N}\\
= & \frac{\alpha_{k+1}}{2(N-k-1)}-\mu_{k}-\frac{\alpha_k^2}{4\tau(N-k)^2}-\frac{\alpha_0}{2N}\\
= & \frac{\alpha_{k}}{2(N-k)} + \frac{\alpha^2_{k}}{4 (N-k)^2 \tau} -\mu_{k}-\frac{\alpha_k^2}{4\tau(N-k)^2}-\frac{\alpha_0}{2N}\\
= & \frac{\alpha_{k}}{2(N-k)} -\mu_{k}-\frac{\alpha_0}{2N}
\end{align*}
and hence by induction on $\min\{i,j\}=0$ all off-diagonal elements are zero and hence the matrix is positive semidefinite.
Hence, an upper bound on the performance of the best subgradient method is given as
\[
  \begin{array}{rl}
    f_{N+1}-f^\star \leq \min  & \tau \gamma^2 +  \frac{NR^2}{2(N+1)^2 \alpha_0} + \sum_{k=0}^{N} \nu_k L^2 \\
    \st & \tau\geq 0, ~\nu\geq 0, ~\alpha\geq 0,\\
                               & \nu_ 0 = \frac{\alpha_0}{2N},\\
                               & \nu_{k+1} \geq \nu_{k} + \frac{\alpha_{k}^2}{4 (N-k)^2 \tau} \quad \forall k \in [0, \dots, N-1]\\
          & \frac{\alpha_{k+1}}{N-k-1} \geq \frac{\alpha_{k}}{N-k} + \frac{\alpha^2_{k}}{2 (N-k)^2 \tau} \quad \forall k\in [0, \dots, N-1].
  \end{array}
\]
After the change of variables $y_k=\tfrac{\alpha_k}{2(N-k)\tau}$ and letting $\nu_k=y_k/2$ for all $k\in [0, \dots, N-1]$ we get
\[
  \begin{array}{rl}
    f_{N+1}-f^\star \leq & \left\{
                           \begin{array}{rl}
                             \min  & \tau \gamma^2 +  \frac{R^2}{4(N+1)^2 y_0\tau} + \tau \sum_{k=0}^{N} y_k L^2 \\
                             \st & \tau\geq 0, ~y_0\geq 0,\\
                                   & y_{k+1} \geq y_{k} + y_{k}^2 \quad \forall k\in [0, \dots, N-1]
                           \end{array}\right.\\
    = \frac{RL}{\sqrt{N+1}}\cdot u^{\mathbb L}_N(\sigma) = &  \left\{
                           \begin{array}{rl}
                             \min  & \sqrt{\tfrac{R^2(\gamma^2 + \sum_{k=0}^{N} y_k L^2)}{((N+1)^2y_0)}}\\
                             \st & y_0\geq 0,\\
                                   & y_{k+1} \geq y_{k} + y_k^2 \quad \forall k\in [0, \dots, N-1]
                           \end{array}\right.                           
  \end{array}
\]
establishing the claim.
\end{appendixproof}

  \begin{figure}
      \centering
      \begin{tikzpicture}[spy using outlines={circle, magnification=2.5, size=4cm, connect spies}]
        \begin{axis}[
          scaled ticks=false,
          tick label style={/pgf/number format/fixed},
          legend style={
          cells={anchor=west},
        },
          grid=major,
          axis lines=left,
          clip=false,
          xlabel={$\theta$},
          ylabel={$\alpha_{\lfloor N \theta \rfloor}^{\mathbb L}(\sigma)$},
          xmin=0,
          xmax=1,
          ymin=0,
          ymax=1,
          ytick={0,0.2,0.4,0.6,0.8,1},
          yticklabels={{$0\cdot\frac{R}{L\sqrt{N+1}}$}, {$0.2\cdot\frac{R}{L\sqrt{N+1}}$}, {$0.4\cdot\frac{R}{L\sqrt{N+1}}$},{$0.6\cdot\frac{R}{L\sqrt{N+1}}$},{$0.8\cdot\frac{R}{L\sqrt{N+1}}$}, {$1\cdot\frac{R}{L\sqrt{N+1}}$}},
          x label style={at={(axis description cs:0.5,0.0)},anchor=north},
          y label style={at={(axis description cs:-.25,.5)},rotate=0,anchor=center, align=right},
          legend pos=outer north east,
          scale=1
          ]

          \addplot[draw=black, mark=none, thick] table [x=k, y=a, col sep=comma] {figures/step-sizes/step-sizes-0.csv};
          \addlegendentry{$\sigma=0$};

          \addplot +[draw=red, mark=none] table [x=k, y=a, col sep=comma] {figures/step-sizes/step-sizes-0.2.csv};
          \addlegendentry{$\sigma=\frac15$};

          \addplot +[draw=green, mark=none] table [x=k, y=a, col sep=comma] {figures/step-sizes/step-sizes-1.0.csv};
          \addlegendentry{$\sigma=1$};

          \addplot +[draw=violet, mark=none] table [x=k, y=a, col sep=comma] {figures/step-sizes/step-sizes-5.0.csv};
          \addlegendentry{$\sigma=5$};

          \addplot[forget plot, draw=violet, densely dotted, thick, domain=0:1, samples=100, variable=theta] {(1-theta)/(5.420362094992213*(1-(1-5.420362094992213^(-2))*theta))};

          \addplot +[draw=brown, mark=none] table [x=k, y=a, col sep=comma] {figures/step-sizes/step-sizes-10.0.csv};
          \addlegendentry{$\sigma=\sqrt{N}$};
      
        \end{axis}        
        \spy on (6.4,0.58) in node [left] at (13.5,1.8);
      \end{tikzpicture}
      \caption{The conic combination $\alpha^{\mathbb L}$ (reindexed in $\theta\in [0, 1)$) proposed in~\cref{prop:u-L-bound} for $N=100$ and various noise levels $\sigma.$ The dashed line corresponds to the step sizes $\alpha_k'$ from~\cref{lem:admissible} for $\sigma=5$.}
      \label{fig:steps}
    \end{figure}
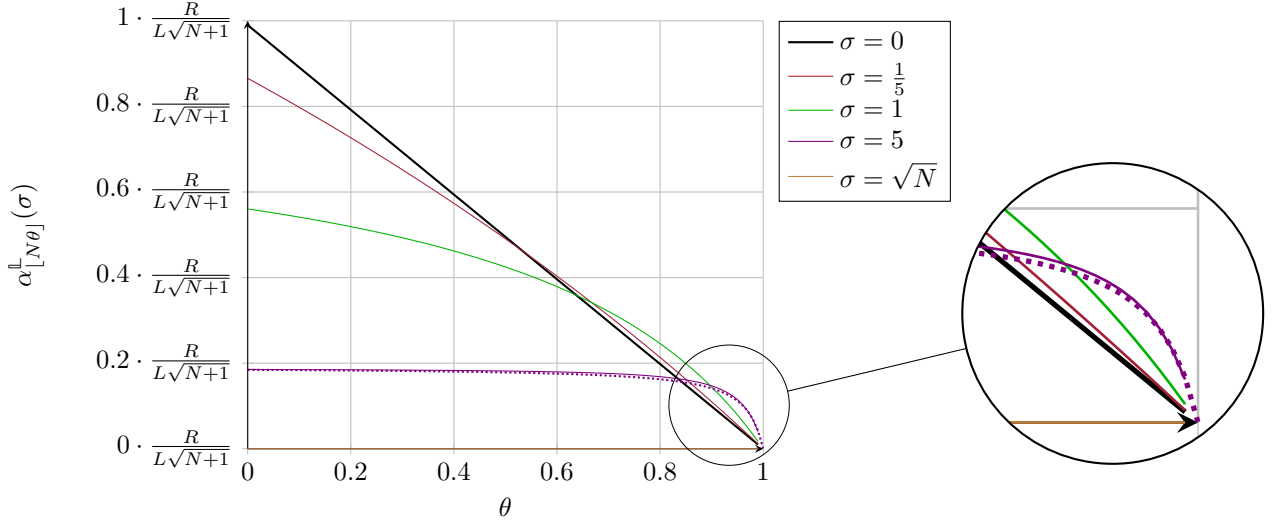

In~\cref{thm:optimal-bound-z} in the appendix, we further characterize the optimal $y_k^*$ from~\cref{prop:u-L-bound}, as well as the resulting performance guarantee. 
\begin{toappendix}
The objective function in the performance optimization problem~\eqref{eq:exact-convex-procedure} is nondecreasing in $y_1,\dots,y_{N-1}$. Therefore, an optimal $y$ can also be found which satisfy the recursion
\begin{equation}\label{eq:y-recursion}
  y_{k+1}=y_k+y_k^2.
\end{equation}
for all $k\in [0, \dots, N-1]$.
The sequence $y_k$ grows monotonously and once it exceeds $1$, its growth becomes doubly-exponential. This can be seen from $y_{k+1}> y_{k}^2$ which implies $y_{k}>y_{0}^{2^k}$, resulting in doubly-exponential growth for $y_{0}>1$. For $0<y_0<1$, we have $y_k\ge y_0+ky_0^2$, so that $y_k$ will eventually exceed $1$ and start its doubly-exponential growth. Finally, we denote with $S_N(y_0)=\sum_{k=0}^Ny_k$ the associated partial sum.

We introduce first the following lemma to study the asymptotics of the  recursion in Equation \eqref{eq:y-recursion}.
\begin{lemma}\label{lem:x-seq-sum}
    For a starting point $y_0\ge0$, we define a sequence $(y_k)_{k\in\mathbb N}$ satisfying recursion \eqref{eq:y-recursion}.
    We have
    \[
    y_0\cdot\sum_{r=0}^{2^N}(k)_ry_0^r\le y_k\le y_0\cdot\sum_{r=0}^{2^N}k^ry_0^r,
    \]
    where $(k)_r=k\cdot(k-1)\cdots(k-r+1)$ is the falling factorial.
    Moreover,
    \[
    \sum_{r=1}^\infty \frac{1}{r}(N+1)_ry_0^r \le S_N(y_0)\le \sum_{r=1}^{\infty} \frac{1}{r}(N+1)^ry_0^r.
    \]
\end{lemma}
\begin{proof}
    We rewrite the recursion to
    \[
    \frac{y_{k+1}}{y_k}=1+y_k,
    \]
    from which we obtain the form
    \[
    y_{k+1}=y_0\prod_{i=0}^{k}(1+y_i).
    \]
    The first two values are given by
    \begin{align*}
        y_1&=y_0+y_0^2,\\
        y_2&=(y_0+y_0^2)\cdot(1+y_0+y_0^2)=y_0+2y_0^2+2y_0^3+y_0^4.
    \end{align*}
    In general, we see that
    \[
    y_k=\sum_{i=1}^{2^k}a_{k,i}y_0^i,
    \]
    for positive integer coefficients $a_{k,i}$, with $a_{k,1}=a_{k,2^k}=1$.
    From $y_{k+1}=y_k(1+y_k)$, we see that the coefficients satisfy the recursion
    \[
        a_{k+1,i}=a_{k,i}+\sum_{j=1}^{i-1}a_{k,j}a_{k,i-j}.
    \]
    For $i=2$, we see that
    \[
        a_{k+1,2}=a_{k,2}+a_{k,1}^2=a_{k,2}+1,
    \]
    which results in $a_{k,2}=k$. Similarly,
    \[
        a_{k+1,3}=a_{k,3}+2a_{k,1}a_{k,2}=a_{k,3}+2k,
    \]
    which results in $a_{k,3}=k(k-1)$.
    We will prove by induction that $a_{k,i}< k^{i-1}$ for $i\ge3$. 
    This holds for $i=3$ and $k\ge2$ since we have $a_{k,3}=k(k-1)<k^2$. Note that this trivially holds for $i>2^k$, since $a_{k,i}=0<k^{i-1}$.  We show that if the induction hypothesis holds for $i\ge3$, it will also hold for $i+1$:
    \begin{align*}
        a_{k+1,i+1}
        =a_{k,i+1}+\sum_{j=1}^{i}a_{k,j}a_{k,i+1-j}
        \le a_{k,i+1}+\sum_{j=1}^{i}k^{j-1}k^{i-j}
        =a_{k,i+1}+i\cdot k^{i-1}.
    \end{align*}
    Repeating this inequality $k-1$ times and applying an integral bound
    \begin{eqnarray*}
        a_{k+1,i+1}
        &\le a_{1,i+1}+i\cdot\sum_{j=1}^kj^{i-1}
        &\le 0+\int_1^{k+1} iz^{i-1}dz\\
        &=[z^i]_1^{k+1}
        &=(k+1)^i-1.
    \end{eqnarray*}
    Hence, for $ky_0<1$, 
    \[
    y_k\le y_0\cdot \sum_{j=0}^\infty(ky_0)^j=y_0\frac{1}{1-ky_0}.
    \]
    For $(N+1)y_0<1$, this leads to the upper bound
    \begin{eqnarray*}
    S_N(y_0)&\le\sum_{k=0}^Ny_0\frac{1}{1-ky_0}
    &<\int_{0}^{(N+1)}\frac{y_0dz}{1-zy_0}\\
    &=-\left[\log(1-zy_0)\right]_0^{(N+1)y_0}
    &=-\log(1-(N+1)y_0)\\
    &=\sum_{r=1}^\infty\frac1r(N+1)^ry_0^r.
    \end{eqnarray*}
    We will similarly prove by induction that $a_{k,i}\ge(k)_{i-1}$.
    It holds (with equality) up to $i=3$. Substituting this into the recursion, we obtain
    \begin{align*}
    a_{k'+1,i+1}-a_{k',i+1}=\sum_{j=1}^ia_{k',j}a_{k',i+1-j}
    \ge\sum_{j=1}^i(k')_{j-1}(k')_{i-j}
    \ge i\cdot (k')_{i-1}.
    \end{align*}
    Summing this inequality from $k'=i-1$ to $k'=k-1$, we obtain
    \begin{eqnarray*}
        a_{k,i+1}&=a_{k,i+1}-a_{1,i+1}&\ge i\sum_{k'=1}^{k-1}(k')_{i-1}\\
        &=i\sum_{k'=i-1}^{k-1}(k')_{i-1}
        &=i!\sum_{k'=i-1}^{k-1}{k'\choose i-1}\\
        &=i!{k\choose i}&=(k)_i.
    \end{eqnarray*}
    where the last step follows from the hockey-stick identity of binomial coefficients.
    Hence,
    \begin{eqnarray*}
    S_N(y)= \sum_{i=1}^{2^N}\sum_{k=0}^Na_{k,i}y^i
    \ge \sum_{i=1}^{2^N}y^i\sum_{k=0}^N(k)_{i-1}
    =\sum_{i=1}^{2^N}\frac{(N+1)_i}{i}y^i
    =\sum_{i=1}^\infty\frac{(N+1)_i}{i}y^i.
    \end{eqnarray*}
\end{proof}

The following theorem presents the optimal step sizes w.r.t. the bound proven in~\cref{prop:u-L-bound}:

\begin{theorem}\label{thm:optimal-bound-z}
Let $y_0$ be the solution of
\begin{equation}\label{eq:y0-implicit}
\sigma^2=y_0S_N'(y_0)-S_N(y_0),
\end{equation}
we have $u^{\mathbb L}_N(\sigma)=\sqrt{S_N'(y_0)}$. Let $y_k$ follow the recursion in Equation \eqref{eq:y-recursion} initialized at $y_0$. The step sizes
\begin{equation}\label{eq:alpha-opt}
\alpha_k^{\mathbb L}=\frac{R(N-k)}{L (N+1)^{3/2}}\frac{y_k}{u^{\mathbb L}_N(\sigma)y_0},
\end{equation}
enjoy the performance
\begin{equation}\label{eq:noisy-performance}
f_{N+1}-f_\star\le\frac{RL}{\sqrt{N+1}}u^{\mathbb L}_N(\sigma).
\end{equation}
\end{theorem}
\begin{proof}
    We want to minimize~\eqref{eq:exact-convex-procedure}.
    Since $y_k$ for $k>0$ is fully determined by $y_0$, we simply need to find the $y_0$ that minimizes
    \[
    \frac{\sigma^2+\sum_{k=0}^Ny_k}{(N+1)y_0}=\frac{\sigma^2+S_N(y_0)}{(N+1)y_0}.
    \]
    We take the derivative w.r.t. $y_0$ and obtain
    \[
    \frac{S_N'(y_0)}{(N+1)y_0}-\frac{\sigma^2+S_N(y_0)}{(N+1)y_0^2}=0,
    \]
    which can be rewritten to~\eqref{eq:y0-implicit}.
    The corresponding performance bound becomes
    \[
    \frac{RL}{\sqrt{N+1}}\sqrt{\frac{\sigma^2+S_N(y_0)}{(N+1)y_0}}.
    \]
    Substituting $S_N'(y_0)=\frac1{y_0}(\sigma^2+S_N(y_0))$, the above can be rewritten to~\eqref{eq:noisy-performance}.
    Similarly substituting $u_N^{\mathbb L}(\sigma)$ into the step sizes from~\cref{prop:u-L-bound} yields~\eqref{eq:alpha-opt}, which completes the proof.
  \end{proof}

\end{toappendix}
In \cref{sec:analysis}, we provide a simpler admissible step size schedule that is asymptotically equivalent to $\alpha_k^{\mathbb L}.$
Figure~\ref{fig:steps} shows the step size schedule from~\cref{prop:u-L-bound} for various $\sigma.$
We see that in the noiseless case, the step sizes coincide with the known optimal step sizes from~\cite{zamani2023exact}. Unsurprisingly, increased gradient corruption begets a less aggressive overall step size schedule. However, the proposed subgradient method is more cautious in earlier iterations than in later ones, where it is, in fact, more aggressive than in the absence of corruption. 

\subsection{Performance lower bounds}
\label{sec:univ-lower-bound}

Consider any algorithm which generates iterates that satisfy the following cone condition
\begin{equation}
  \label{eq:cone_condition}
  \begin{split}
    x_k = &~x_0 - \text{cone}(g_0+e_0, \dots, g_{k-1}+e_{k-1
            }) \quad \forall k \in [1, \dots, N],\\
    x_{N+1}= &~ x_0 - \text{cone}(g_0+e_0, \dots, g_{N-1}+e_{N-1}).
  \end{split}
\end{equation}
This class of algorithms includes any subgradient method with non-negative step sizes, the Nesterov accelerated gradient descent method and most practically relevant variable step size algorithms. Intuitively, it captures any algorithm which moves into the negative of the (noisy) subgradients observed up to that point.

We now propose a lower bound on the performance of any method which satisfies Equation \eqref{eq:cone_condition} by choosing the initial condition $x_0-x_\star$, subgradients $g_0, \dots, g_{N+1}$ and noise vectors $e_0, \dots, e_{N-1}$ adversarially. As is standard in performance estimation optimization \citep{Drori_Teboulle_2016, taylor_smooth_2017}, we will do so implicitly by considering its Grammian matrix. As the name suggests, this Grammian encodes all inner products between the variables of interest as entries in a symmetric positive semidefinite matrix $G$. For notational convenience, we will write 
\(
  G(x_0-x_\star, g_i) \defn \langle x_0-x_\star, g_i \rangle
\)
to denote the entry related to the inner product between the initial condition $x_0-x_\star$ and the gradient $g_i$.

\begin{theorem}\label{thm:lower-matrix}
    For any optimization algorithm which satisfies \eqref{eq:cone_condition} there is a function $f\in \mc F$ so that
    \begin{equation}
      \label{eq:universal_lb}
      f_{N+1}-f_\star\geq
      \left\{
        \begin{array}{rl}
          \max_{G\succeq 0, \Delta\geq 0}  & \Delta\\
          \st & G(g_j, g_i) + G(e_j, g_i)  = 0 \quad \forall (i, j) \in [0, \dots, N+1]\times [0, \dots, N-1]:  j< i \\
                                           & G(g_j, g_i) + G(e_j, g_i) \geq 0 \quad (i, j) \in [0, \dots, N+1]\times [0, \dots, N-1]:  j\geq i\\
                                           & G(x_0-x_\star, g_i)= \Delta\quad \forall i\in [0, \dots, N+1]\\
                                           & G(x_0-x_\star, x_0-x_\star) \leq R^2\\
                                           & G(g_i, g_i) \leq L^2\quad\forall i\in [0, \dots, N+1]\\
                                           & \sum_{i=0}^{N-1} G(e_i, e_i) \leq \gamma^2.
                                             
        \end{array}
      \right.
    \end{equation}
\end{theorem}
\begin{appendixproof}[Proof of Proposition \ref{thm:lower-matrix}.]

  We remark that the Grammian must satisfy $\sum_{i=0}^{N-1} G(e_i, e_i) \defn \sum_{i=0}^{N-1} \langle e_i, e_i\rangle\leq \gamma^2$ as the power of the adversary is bounded. 
  
  Suppose now that we find a Grammian that additionally satisfies the conditions
  \begin{align}
    & G(g_j, g_i) + G(e_j, g_i) = \langle g_j +e_j, g_i \rangle = 0 \quad (i, j) \in [0, \dots, N+1]\times [0, \dots, N-1]:  j< i \label{eq:lb_cond_G_1}\\
    & G(g_j, g_i) + G(e_j, g_i) = \langle g_j +e_j, g_i \rangle \geq 0 \quad (i, j) \in [0, \dots, N+1]\times [0, \dots, N-1]:  j\geq i \label{eq:lb_cond_G_2}\\
    \intertext{and}
    & G(x_0-x_\star, g_i) = \langle x_0 - x_\star, g_i \rangle = \Delta\quad \forall i\in [0, \dots, N+1] \label{eq:lb_cond_G_3}
  \end{align}
  for some constant $\Delta\geq 0$. We claim that this implies that the suboptimality gap of any method satisfying Equation \eqref{eq:cone_condition} is at least $\Delta$. Additionally, we set $f^\star=0$ and $f_k=f(x_k)=\Delta$ for all $k\in [0, \dots, N+1]$. We need to verify the fact that the adversarially chosen initial condition $x_0-x_\star$, subgradients $g_0, \dots, g_{N+1}$ errors $e_0, \dots, e_{N-1}$ and associated function values $f_0, \dots, f_{N+1}$ are indeed compatible with the considered function class $\mc F$, i.e., the condition 
  \begin{equation}
    \label{eq:interpolation}
    f\in \mc F, \quad f_k=f(x_k) \quad {\rm{and}}\quad g_k\in \partial f(x_k) \qquad \forall k\in [\star, 0, \dots, N+1]
  \end{equation}
  holds.
  A well known result \citep{boyd2004convex, Drori_Teboulle_2016} is that this infinite dimensional interpolation condition can be reduced to a finite system of subgradient inequalities
  \[
    \eqref{eq:interpolation}\iff
    \left\{
      \begin{array}{l@{\quad}l}
        \multicolumn{2}{l}{g_\star=0, ~\langle x_0-x_\star, x_0-x_\star\rangle\leq R^2,} \\
        \langle g_k, g_k\rangle \leq L^2 & \forall k \in [0, \dots, N+1], \\
        
        f_j \geq f_i + \langle g_i, x_j -x_i\rangle & \forall i, j \in [\star, 0, \dots, N+1].
      \end{array}
    \right.
  \]
  It remains now to verify these conditions as the claim follows immediately from optimizing over all $G$ satisfying the stated conditions.
  
  First, we find that for all $i\in [0, \dots, N+1]$ we have
  \begin{align*}
    &f_i-f_\star \geq \langle x_i - x_\star, g_i \rangle  \\
    \iff &\Delta \geq \langle x_0 - x_\star, g_i \rangle -  \langle \text{cone}(g_0+e_0, \dots, g_{i-1 \wedge N-1}+e_{i-1\wedge N-1
           }), g_i\rangle=\Delta.
  \end{align*}
  Here the first equivalence follows from condition (\ref{eq:cone_condition}) and $f_i=\Delta$ for all $i$ and $f^\star=0$. The last equality follows~\eqref{eq:lb_cond_G_1} and~\eqref{eq:lb_cond_G_3}.
  Second, we verify convexity by
  \begin{align*}
    & f_j\geq f_i + \langle g_i, x_j-x_i\rangle\\
    \iff & 0\geq \langle g_i, x_j-x_i\rangle\\
    \iff & 0\geq -\langle g_i, \text{cone}(g_0+e_0, \dots, g_{j-1\wedge N-1}+e_{j-1\wedge N-1
           }) + \langle g_i, \text{cone}(g_0+e_0, \dots, g_{i-1\wedge N-1}+e_{i-1\wedge N-1
           })\rangle\\
    \impliedby & 0\geq -\langle g_i, \text{cone}(\{ g_k+e_k : \forall k \in [i, \dots, j-1\wedge N-1]\}) \rangle\geq 0
  \end{align*}
  for all $i\in [0,\dots, N+1]$ and $j$ in $[0, \dots, N+1]$. The first equivalence follows from our choice $f_k=\Delta$ for all $k\in [0, \dots, N+1]$. The second equivalence follows from the condition (\ref{eq:cone_condition}). The third equivalence is a result of conditions \eqref{eq:lb_cond_G_1} and \eqref{eq:lb_cond_G_2}.
\end{appendixproof}

The previous result gives a lower bound on the performance of any algorithm satisfying Equation (\ref{eq:cone_condition}) in the form of a tractable convex semidefinite optimization problem. The following results will make the discussed lower bound much more explicit. Let $\nu\in [0, 1]$ be the unique solution of
\begin{equation}
  \label{def:nu}
  \sum_{k=0}^{N-1}  \frac{(N-k)\nu^2}{1+(N-(k+1))\nu} = \sigma^2,
\end{equation}
and introduce the increasing sequence
$$\gamma^2_k \defn L^2 \frac{(N-k)\nu^2}{1+(N-(k+1))\nu}\in [0, L^2)$$ for all $k\in [0, \dots, N-1]$.
The following result gives a lower bound on the performance by considering only an adversaries which corrupts the subgradients by allocating their budget as $\norm{e_k}^2 = \gamma_k^2$ for all $k\in [0, \dots, N-1]$.

\begin{corollary}\label{cor:lower}
  Any optimization algorithm which satisfies \eqref{eq:cone_condition} has the performance lower bound
  \[
    f_{N+1}-f_\star\geq \frac{RL}{\sqrt{N+1}}\cdot \ell_N(\sigma),
  \]
  where $\ell_N(\sigma) = \sqrt{1+N \nu}$ with $\nu\in[0,1]$ the unique solution of Equation~\eqref{def:nu}. 
\end{corollary}
\begin{appendixproof}[Proof of~\cref{cor:lower}]

  We consider candidate Grammian matrices $G=(A , B^\top ; B ,C)$ of the form
  \[
    \left(\begin{array}{ccccccc|ccccc}
      R^2 & \star  & \star  & \star  & \dots & \star  & \star & \star & \star & \star & \dots & \star  \\
      F \nu L^2  & L^2 & \star & \star  & \dots & \star & \star & \star & \star & \star& \dots & \star  \\
      F \nu L^2 & \nu L^2 & L^2 & \star & \dots & \star & \star & \star & \star & \star & \dots & \star   \\
      F \nu L^2 & \nu L^ 2 & \nu L^2 & L^2 &\dots  &  \star &\star  & \star & \star & \star & \dots & \star \\
      \vdots & \vdots & \vdots & \vdots &  \ddots  & \vdots & \vdots &\vdots &\vdots & \vdots & \ddots & \vdots \\
      F \nu L^2 & \nu L^ 2 & \nu L^ 2 & \nu L^ 2 & \dots & L^2 & L^2 & \star & \star & \star & \dots & \star \\
      F \nu L^2 & \nu L^ 2 & \nu L^ 2 & \nu L^ 2 & \dots & L^2 & L^2 & \star & \star & \star & \dots & \star \\
      \hline
      - F\gamma_{0}^2  & - \gamma_{0}^2  & - \nu L^ 2 & - \nu L^ 2 & \dots & - \nu L^ 2 & - \nu L^ 2 & \gamma_{0}^2  & \star & \star  & \dots & \star \\
      - F\gamma_{1}^2  &  - \gamma_{1}^2  & - \gamma_{1}^2  & - \nu L^ 2& \dots  & - \nu L^ 2 & - \nu L^ 2 & \gamma_{1}^2  & \gamma_{1}^2  & \star & \dots & \star \\
      -F\gamma_{2}^2 & -\gamma_{2}^2  & -\gamma_{2}^2  & -\gamma_{2}^2 & \dots & - \nu L^2 & - \nu L^2  &  \gamma_{2}^2    & \gamma_{2}^2  & \gamma_{2}^2  & \dots & \star \\
      \vdots & \vdots & \vdots & \vdots & \ddots  & \vdots  & \vdots &\vdots &  \vdots & \vdots & \ddots & \vdots \\
      - F\gamma_{N-1}^2 L^2 & -\gamma_{N-1}^2 & -\gamma_{N-1}^2 & -\gamma_{N-1}^2 & \dots & -\nu L^2 & -\nu L^2 & \gamma_{N-1}^2 & \gamma_{N - 1}^2 & \gamma_{N-1}^2 & \dots & \gamma_{N - 1}^2  \\
    \end{array}\right)
\]
for $F\geq 0$ where here [$\star$] indicate symmetric entries which are omitted for the sake of brevity.

We now check whether the matrix $G$ is indeed feasible in Equation (\ref{eq:universal_lb}) for some $\Delta\geq 0$. Observe first that
\begin{equation*}
G(g_j, g_i) + G(e_j, g_i)  = \nu L^2- \nu L^2=0 \quad \forall (i, j) \in [0, \dots, N+1]\times [0, \dots, N-1]:  j< i
\end{equation*}
and
\begin{equation*}
  G(g_j, g_i) + G(e_j, g_i) = L^2 - \gamma_j^2 \geq 0 \quad \forall (i, j) \in [0, \dots, N+1]\times [0, \dots, N-1]:  j\geq i.\\
\end{equation*}
Second, we clearly have
\begin{equation*}
 G(x_0-x_\star, g_i)= F \nu L^2 =: \Delta \geq 0 \quad \forall i\in [0, \dots, N+1]\\
\end{equation*}
and $G(x_0-x_\star, x_0-x_\star)= R^2$, $G(g_i, g_i) = L^2$ for all $i\in [0, \dots, N]$ and $\sum_{i=0}^{N-1} G(e_i, e_i) = \sum_{i=0}^N \gamma^2_i = \gamma^2$ from Equation (\ref{def:nu}).

It finally remains to verify that the candidate Grammian $G$ is indeed positive semidefinite. From Schur's complement it suffices to verify that $C\succ 0$ and $A-B^\top C^{-1} B\succeq 0$. We establish that $C$ is positive definite in Lemma \ref{lemma:psd}. From Lemma \ref{lemma:inverse} it follows immediately that $S = A-B^\top C^{-1} B$ is identically zero outside a $2\times2$ block in the top left corner where it takes on the values
\[
  \begin{pmatrix}
    S_{11} & \star \\
    S_{2, 1} & S_{2,2}
  \end{pmatrix}
  =
  \begin{pmatrix}
    R^2-F^2\gamma_0^2 & \star \\
    F\nu L^2-F\gamma_0 & L^2-\gamma_0^2
  \end{pmatrix}.
\]
Hence, as we have here that $\gamma_0^2<L^2$ it follows that for $S\succeq 0$ it suffices to have $S_{11} = S_{2,1}^2/S_{2,2}$. This leads to
\begin{align*}
  \iff & \left(R^2-\frac{F^2 N \nu^2 L^2}{1+(N-1)\nu}\right)\left(L^2-\frac{N \nu^2 L^2}{1+(N-1)\nu}\right)=F^2 \left( \nu L^2-\frac{N \nu^2 L^2}{1+(N-1)\nu}\right)^2\\
  \iff & \left(R^2-\frac{F^2 N \nu^2 L^2}{1+(N-1)\nu}\right)\left(1-\frac{N \nu^2}{1+(N-1)\nu}\right)=F^2 L^2 \left( \nu -\frac{N \nu^2}{1+(N-1)\nu}\right)^2\\
  \iff &  R^2 \left(1-\frac{N \nu^2}{1+(N-1)\nu}\right)=F^2 L^2 \left(  \left( \nu -\frac{N \nu^2}{1+(N-1)\nu}\right)^2 + \frac{N \nu^2 }{1+(N-1)\nu} \left(1-\frac{N \nu^2}{1+(N-1)\nu}\right) \right)\\
  \iff & R^2 L^2 \left(1-\frac{N \nu^2}{1+(N-1)\nu}\right)=F^2 L^4 \nu^2 \left( 1-\frac{2\nu N}{1+(N-1)\nu}+\frac{N}{1+(N-1)\nu} \right)\\
  \iff & R^2 L^2 \left(1+(N-1)\nu-N \nu^2\right)=F^2 L^4 \nu^2 \left( 1+(N-1)\nu-2\nu N+N \right)\\
  \iff & F^2 L^4 \nu^2 = \frac{R^2L^2 \left(1+(N-1)\nu-N \nu^2\right)}{N+1-\nu(N+1)}\\
  \iff & \Delta = F L^2 \nu = RL \frac{\sqrt{1+(N-1)\nu-N \nu^2}}{\sqrt{N+1-\nu(N+1)}}=RL\frac{\sqrt{1+N\nu}}{\sqrt{N+1}}
\end{align*}
from which the claim follows.
\end{appendixproof}

\begin{toappendix}
\begin{lemma}
  \label{lemma:psd}
  Let $c_0>c_1>\dots>c_{N-1}>0$ for any $N\geq 1$. Then
  \[
   \left(\begin{array}{ccccc}
       c_{0}^2  & \star & \star  & \dots & \star \\
       c_{1}^2 & c_{1}^2   & \star & \dots & \star \\
        c_{2}^2     & c_{2}^2   & c_{2}^2   & \dots & \star \\
       \vdots &  \vdots & \vdots & \ddots & \vdots \\
       c_{N-1}^2   & c_{N-1}^2   & c_{N-1}^2   & \dots & c_{N-1}^2    \\
   \end{array}\right)\succ 0.
 \]
\end{lemma}
\begin{proof}
  We will show this result by induction.
  Clearly the result holds for $N=1$. Suppose now the results hold for all $N=k$. Then, consider the $N=k+1$ and partition the matrix of interest as follows
  \[
    \begin{pmatrix}
      A & B\\
      B^\top & C
    \end{pmatrix}
    =
    \left(\begin{array}{ccccc|c}
      c_{0}^2  & \star & \star  & \dots & \star & \star \\
      c_{1}^2 & c_{1}^2   & \star & \dots & \star & \star\\
      c_{2}^2     & c_{2}^2   & c_{2}^2   & \dots & \star & \star \\
      \vdots &  \vdots & \vdots & \ddots & \vdots & \star\\
      c_{k-1}^2   & c_{k-1}^2   & c_{k-1}^2   & \dots & c_{k-1}^2  & \star  \\
      \hline
      c_{k}^2   & c_{k}^2   & c_{k}^2   & \dots & c^2_{k} & c_{k}^2    \\
    \end{array}\right).
\]
From the induction hypothesis we know that $A\succ 0$. From Schur's complement it now suffices that
\[
  S = C-B^\top \left(A^{-1} B\right) =c_k^2- B^\top \begin{pmatrix} 0 \\ \vdots \\ 0 \\ \frac{c_k^2}{c^2_{k-1}} \end{pmatrix}  = c_k^2 - c_k^4/c_{k-1}^2>0\iff c^2_{k}>c^2_{k-1}
\]
to claim that indeed the result also holds for $N=k+1$. 
The expression for $A^{-1}B$ above is obtained by solving $Ab=B$.
\end{proof}

\begin{lemma}
  \label{lemma:inverse}
  Consider $B,C$ as given in~\cref{cor:lower}.
  We have
  \[
    B C^{-1} B =
    \left(\begin{array}{ccccccc}
      F^2\gamma_0^2 & \star  & \star  & \star  & \dots & \star  & \star  \\
      F\gamma_0^2  & \gamma_0^2 & \star & \star  & \dots & \star & \star   \\
      F \nu L^2 & \nu L^2 & L^2 & \star & \dots & \star & \star    \\
      F \nu L^2 & \nu L^ 2 & \nu L^2 & L^2 &\dots  &  \star &\star   \\
      \vdots & \vdots & \vdots & \vdots &  \ddots  & \vdots & \vdots  \\
      F \nu L^2 & \nu L^ 2 & \nu L^ 2 & \nu L^ 2 & \dots & L^2 & L^2  \\
      F \nu L^2 & \nu L^ 2 & \nu L^ 2 & \nu L^ 2 & \dots & L^2 & L^2  
         \end{array}
        \right).
  \]
\end{lemma}
\begin{proof}
  We first show that
  \begin{align*}
    & \left(\begin{array}{ccccc}
      \gamma_{0}^2  & \star & \star  & \dots & \star \\
      \gamma_{1}^2 & \gamma_{1}^2 & \star & \dots & \star \\
      \gamma_{2}^2  & \gamma_{2}^2 & \gamma_{2}^2 & \dots & \star \\
      \vdots &  \vdots & \vdots & \ddots & \vdots \\
      \gamma_{N\!-\!1}^2 & \gamma_{N\!-\!1}^2 & \gamma_{N\!-\!1}^2 & \dots & \gamma_{N\!-\!1}^2   \\
    \end{array}\right)
      \begin{pmatrix}
        -F & -1 & -\frac{1}{\nu}- (N-1) & 0 &  & 0 & 0 \\
        0& 0 & \frac{1}{\nu}+ (N-2) & -\frac{1}{\nu}- (N-2)& & 0& 0 \\
        0& 0 & 0 & \frac{1}{\nu}+ (N-3) & &  0 & 0 \\
        \vdots & \vdots & \vdots & \vdots & \ddots&  \vdots & \vdots\\
        0 & 0 & 0 &0 &  & -\frac{1}{\nu} & -\frac{1}{\nu}
      \end{pmatrix}\\
    =&
       \begin{pmatrix}
         - F\gamma_{0}^2  & - \gamma_{0}^2  & - \nu L^ 2 & - \nu L^ 2 & \dots & - \nu L^ 2 & - \nu L^ 2 \\
         - F\gamma_{1}^2  &  - \gamma_{1}^2  & - \gamma_{1}^2  & - \nu L^ 2& \dots  & - \nu L^ 2 & - \nu L^ 2 \\
         -F\gamma_{2}^2 & -\gamma_{2}^2  & -\gamma_{2}^2  & -\gamma_{2}^2 & \dots & - \nu L^2 & - \nu L^2  \\
         \vdots & \vdots & \vdots & \vdots & \ddots  & \vdots  & \vdots \\
         - F\gamma_{N-1}^2 & -\gamma_{N-1}^2 & -\gamma_{N-1}^2 & -\gamma_{N-1}^2 & \dots & -\nu L^2 & -\nu L^2  \\
       \end{pmatrix} \\
    = &  B  =:[B_k]_{k\in [\star, 0, \dots, N+1]} =: [B_{i, k}]_{i\in [0, \dots, N-1], k\in [\star, 0, \dots, N+1]}.
  \end{align*}
  Clearly, we have that
  \(
  B_{N} = B_{N+1} = -\tfrac{C \delta_N}{\nu}= \tfrac{\gamma^2_{N-1}}{\nu} 1_N= -\nu L^2 1_N 
  \)
  where $1_N$ and $\delta_N$ denotes the vector of all ones and the $N$-th unit vector, respectively, using here that $\gamma^2_{N-1}=\nu^2 L^2$.
  Similarly, we also have
  \(
  B_{\star} =  -C F \delta_1 = -F [\gamma^2_0, \dots, \gamma^2_{N-1}]
  \)
  and
  \(
  B_0 = -C \delta_1 = - [\gamma^2_0, \dots, \gamma^2_{N-1}] .
  \)
  For any $k\in [1, \dots, N-1]$ and $i\geq k$ we observe that
  \[
    B_{i, k} = \gamma^2_{i} \left(-\frac 1\nu-(N-(k-1))\right)+\gamma^2_i \left(\frac 1\nu+(N-k)\right) = -\gamma_i^2
  \]
  whereas for any $k\in [1, \dots, N-1]$ and $i<k$ we have from Lemma \ref{eq:gamma-identity} that
  \[
    B_{i, k} = \gamma_{k-1}^2 \left(-\frac{1}{\nu}-(N-k)\right)+\gamma^2_{k} \left( \frac{1}{\nu}+(N-(k+1)) \right) =-\nu L^2.
  \]
  Finally, via straightforward algebraic manipulation we have that
  \begin{align*}
    & B^\top \left(C^{-1} B\right)\\
    = & \begin{pmatrix}
      - F\gamma_{0}^2 & - F\gamma_{1}^2 & -F\gamma_{2}^2 & \hdots  & - F\gamma_{N-1}^2\\
      - \gamma_{0}^2  & - \gamma_{1}^2 &  -\gamma_{2}^2 & \hdots & -\gamma_{N-1}^2 \\
      - \nu L^ 2 & - \gamma_{1}^2 & - \gamma_{2}^2 & \hdots &  -\gamma_{N-1}^2 \\
      - \nu L^ 2 & - \nu L^ 2 & -\gamma_{2}^2 & \hdots & -\gamma_{N-1}^2 \\
      \vdots & \vdots & \vdots &\ddots&\vdots  \\
      - \nu L^ 2 & - \nu L^ 2 & - \nu L^2 & \dots & -\nu L^2 \\
      - \nu L^ 2 & - \nu L^ 2& - \nu L^2 & \dots & -\nu L^2\\
    \end{pmatrix}
        \begin{pmatrix}
        -F & -1 & -\frac{1}{\nu}- (N-1) & 0 &  & 0 & 0 \\
        0& 0 & \frac{1}{\nu}+ (N-2) & -\frac{1}{\nu}- (N-2)& & 0& 0 \\
        0& 0 & 0 & \frac{1}{\nu}+ (N-3) & &  0 & 0 \\
        \vdots & \vdots & \vdots & \vdots & \ddots&  \vdots & \vdots\\
        0 & 0 & 0 &0 &  & -\frac{1}{\nu} & -\frac{1}{\nu}
    \end{pmatrix}\\
    = &
         \left(\begin{array}{ccccccc}
      F^2\gamma_0^2 & \star  & \star  & \star  & \dots & \star  & \star  \\
      F\gamma_0^2  & \gamma_0^2 & \star & \star  & \dots & \star & \star   \\
      F \nu L^2 & \nu L^2 & L^2 & \star & \dots & \star & \star    \\
      F \nu L^2 & \nu L^ 2 & \nu L^2 & L^2 &\dots  &  \star &\star   \\
      \vdots & \vdots & \vdots & \vdots &  \ddots  & \vdots & \vdots  \\
      F \nu L^2 & \nu L^ 2 & \nu L^ 2 & \nu L^ 2 & \dots & L^2 & L^2  \\
      F \nu L^2 & \nu L^ 2 & \nu L^ 2 & \nu L^ 2 & \dots & L^2 & L^2  
         \end{array}
        \right)
  \end{align*}
  with the help of Lemma \ref{eq:gamma-identity} (to derive the entries resulting in $\nu L^2$) as well as Lemma \ref{eq:gamma-identity-2} (to derive the entries equal to $L^2$).
\end{proof}

\begin{lemma}
  \label{eq:gamma-identity}
  We have
  \[
    \gamma_{k-1}^2 \left(-\frac{1}{\nu}-(N-k)\right)+\gamma^2_{k} \left( \frac{1}{\nu}+(N-(k+1)) \right) =-\nu L^2
  \]
  for all $k\in [1, \dots, N-1]$.
\end{lemma}
\begin{proof}
  Observe
  \begin{align*}
    & \gamma_{k-1}^2 \left(-\frac{1}{\nu}-(N-k)\right)+\gamma^2_{k} \left( \frac{1}{\nu}+(N-(k+1)) \right)\\
    = & L^2 \frac{(N-k+1)\nu^2}{1+(N-k)\nu} \left(-\frac{1}{\nu}-(N-k)\right)+L^2 \frac{(N-k)\nu^2}{1+(N-(k+1))\nu} \left( \frac{1}{\nu}+(N-(k+1)) \right)\\
    = & \frac 1\nu L^2 \left(\frac{(N-k+1)\nu^2}{1+(N-k)\nu} \left(-1-(N-k)\nu\right)  +\frac{(N-k)\nu^2}{1+(N-(k+1))\nu} \left( 1+(N-(k+1))\nu \right) \right)\\
    = & \frac 1\nu L^2 \left(-(N-k+1)\nu^2  + (N-k)\nu^2\right)=-\nu L^2
  \end{align*}
  for all $k\in [1, \dots, N-1]$.
\end{proof}

\begin{lemma}
  \label{eq:gamma-identity-2}
  We have
  \[
    -\nu L^2 \left(-\frac{1}{\nu}-(N-k)\right) - \gamma^2_{k} \left( \frac{1}{\nu}+(N-(k+1)) \right) = L^2
  \]
  for all $k\in [0, \dots, N-1]$.
\end{lemma}
\begin{proof}
  Observe
  \begin{align*}
    & -\nu L^2 \left(-\frac{1}{\nu}-(N-k)\right)-\gamma^2_{k} \left( \frac{1}{\nu}+(N-(k+1)) \right)\\
    = & L^2 \left(1+(N-k)\nu\right)-L^2 \frac{(N-k)\nu^2}{1+(N-(k+1))\nu} \left( \frac{1}{\nu}+(N-(k+1)) \right)\\
    = &  L^2 \left(1+(N-k)\nu\right)- \frac 1\nu L^2 \left(\frac{(N-k)\nu^2}{1+(N-(k+1))\nu} \left( 1+(N-(k+1))\nu \right) \right)\\
    = & L^2 \left(1+(N-k)\nu\right)- \frac 1\nu L^2 (N-k)\nu^2= L^2
  \end{align*}
  for all $k\in [0, \dots, N-1]$.
\end{proof}

\end{toappendix}

When $\sigma=0$, then the solution of~\eqref{def:nu} is $\nu=0$, which leads to the known bounds~\eqref{eq:classical-stepsize}. This lower bound also matches the upper bound in Proposition \ref{prop:u-L-bound}, with $y_k=0$ for all $k$, and we obtain $\alpha_k^{\mathbb L}=\alpha_k^*$ from~\eqref{eq:zamani-stepsize}.
Proposition \ref{prop:u-L-bound} implies in fact that the entire manifold of stepsizes $h\in \mc H(\alpha^\star)$ given in \eqref{eq:zamani-stepsize} enjoys the same optimal worst-case suboptimality gap.

When $\sigma = \sqrt{N}$, then $\nu=1$, which corresponds to the trivial upper bound
\[
  f_{N+1}-f_\star = f_0 - f_\star \leq RL,
\]
This agrees with the observation that no progress is possible since the adversary can maximally corrupt the subgradient $\tilde g_k = g_k+e_k=0$ for $k\in [0, \dots, N-1]$. In this regime, an optimal solution of the performance optimization problem \eqref{eq:gen_pop} is given by  $\alpha=0$, $\nu_\star=\frac12$, $\nu_k=\frac12(N+1)^{-1}$ for all $k\in [0, \dots, N]$ and $\tau=0$, resulting in a matching upper bound.
This suboptimality upper bound is attained with equality for the function $f(x)=L|x|$ with $x_0=R$.

In~\cref{sec:analysis}, we study the asymptotics of $\ell_N(\sigma)$ for $0<\sigma<\sqrt N$ and compare it to the asymptotics of $u^{\mathbb S}_N(\sigma)$ and $u^{\mathbb L}_N(\sigma)$.

\section{Analysis}\label{sec:analysis}
In this section, we analyze the upper and lower bounds on the suboptimality gap presented in~\cref{prop:u-L-bound} and~\cref{cor:lower} for $\sigma\in(0,L\sqrt N)$. 
Figure~\ref{fig:rel_gap} shows the relative worst-case suboptimality performance gap 
\begin{equation}
  \label{eq:relative-gap}
  \max_{\sigma\in [0, \sqrt{N}]}\frac{u_N^{\{\mathbb S, \mathbb L\}}(\sigma)-\ell_N(\sigma)}{\ell_N(\sigma)}
\end{equation}
between our upper and lower bounds on the worst-case suboptimality gap. We see that this relative difference never exceeds $1\%$. That is, the best subgradient method found by either \cref{thm:main-upper-bound} and \cref{prop:u-L-bound} can not be (significantly) improved by any algorithm satisfying the cone condition in Equation~\eqref{eq:cone_condition}. In the remainder of this section, we analyze this difference also via analytical techniques.
In the following lemma, we present an explicit step size schedule that is admissible w.r.t.~\cref{prop:u-L-bound} and provides an upper bound for $u_N^{\{\mathbb S, \mathbb L\}}$.


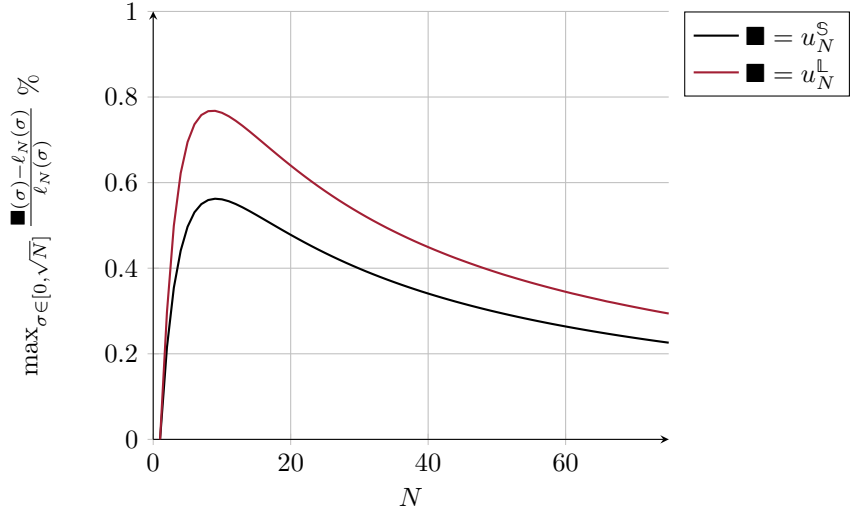
\begin{figure}[]
  \centering
  \begin{tikzpicture}
    \begin{axis}[
      scaled ticks=true,
      grid=major,
      axis lines=left,
      clip=false,
      xlabel={$N$},
      ylabel={$\max_{\sigma\in [0,\sqrt{N}]}\frac{\blacksquare(\sigma)-\ell_N(\sigma)}{\ell_N(\sigma)}~\%$},
      xmin=0,
      xmax=75,
      ymin=0,
      ymax=1,
      x label style={at={(axis description cs:0.5,0.0)},anchor=north},
      y label style={at={(axis description cs:0,.5)},rotate=0,anchor=south, align=right},
      legend pos=outer north east,
      scale=1
      ]

      \addplot[mark=none, thick] table [x=N, y=rel_gap_sdo, col sep=comma] {figures/rel_gap/rel_gap.csv};
      \addlegendentry{$\blacksquare=u^{\mathbb S}_N$};

      \addplot +[mark=none, thick] table [x=N, y=rel_gap_soc, col sep=comma] {figures/rel_gap/rel_gap.csv};
      \addlegendentry{$\blacksquare=u^{\mathbb L}_N$};      

    \end{axis}        
  \end{tikzpicture}
  \caption{Relative difference between upper and lower worst-case performance bounds as a function of the number of steps $N$. For every $N$, the shown value is the maximum of the relative difference over $\sigma\in[0,\sqrt N]$. Theorem \ref{thm:squeeze} implies that this difference decays asymptotically to zero at rate at least $\mc O(\log(N)/N)$.}
  \label{fig:rel_gap}
\end{figure}

We denote the \emph{generalized harmonic numbers} by
\[
H_m(a)=\sum_{k=0}^{m-1}\frac1{a+k},
\]
so that $H_m(1)$ corresponds to the $m$-th harmonic number.
\begin{lemma}\label{lem:admissible}
    Consider the step sizes
    \[
    \alpha_k'=\frac{R}{L}\frac{N-k}{(N+1)^{3/2}}\frac{u(\sigma)}{u(\sigma)^2-(u(\sigma)^2-1)\frac{k}{N+1}}\sqrt{\frac{\sigma^2+2\log u(\sigma)}{\sigma^2+H_{N+1}\left(1+\frac{N+1}{u(\sigma)^2-1}\right)}}
    \]
    where the function $u(\sigma)\ge1$ is defined implicitly as the solution in
    \begin{equation}\label{eq:lambda-implicit}
        \sigma^2=u^2-1-2\log u.
    \end{equation}
    Any subgradient method $h\in \mc H(\alpha')$ enjoys the worst-case suboptimality gap
    \[
    f_{N+1}-f_\star \le \frac{RL}{\sqrt{N+1}}\cdot u(\sigma).
    \]
    Furthermore,
    \begin{equation}\label{eq:negligible-factor}
    1\le  \xi_N(\sigma) = \sqrt{\frac{\sigma^2+2\log u(\sigma)}{\sigma^2+H_{N+1}\left(1+\frac{N+1}{u(\sigma)^2-1}\right)}}\le 1+\frac{2}{N},
    \end{equation}
    for any $\sigma\in[0,\sqrt{N}]$.
\end{lemma}
\begin{appendixproof}[Proof of~\cref{lem:admissible}]
    Note that for any $y_0'>0$, the sequence $y_k'=y_0'/(1-ky_0')$ satisfies $y_{k+1}'>y_k'+(y_k')^2$, so that it satisfies the conditions of~\cref{prop:u-L-bound}.
    We pick
    \[
    y_0'=\frac{1-u(\sigma)^{-2}}{N+1},
    \]
    so that
    \[
    y_k'=\frac{1-u(\sigma)^{-2}}{N+1-k(1-u(\sigma)^{-2})}=\frac{1}{\frac{u(\sigma)^2+N}{u(\sigma)^2-1}+N-k}.
    \]
    Hence,
    \[
    \sum_{k=0}^Ny_k'=\sum_{k=0}^N\frac{1}{\frac{u(\sigma)^2+N}{u(\sigma)^2-1}+k}=H_{N+1}\left(\frac{u(\sigma)^2+N}{u(\sigma)^2-1}\right)=H_{N+1}\left(1+\frac{N+1}{u(\sigma)^2-1}\right),
    \]
    for the generalized harmonic number $H_m=\sum_{k=0}^{m-1}(a+m)^{-1}$. This leads to the factor
    \[
    u_N'=\sqrt{\frac{\sigma^2+H_{N+1}(1+\frac{N+1}{u(\sigma)^2-1})}{(N+1)y_0}}=\sqrt{\frac{\sigma^2+H_{N+1}(1+\frac{N+1}{u(\sigma)^2-1})}{1-u(\sigma)^{-2}}}=u(\sigma)\sqrt{\frac{\sigma^2+H_{N+1}(1+\frac{N+1}{u(\sigma)^2-1})}{u(\sigma)^2-1}}.
    \]
    By an integral bound, the generalized harmonic numbers are bounded by
    \[
    H_m(a)\le\log\left(1+\frac m{a-1}\right).
    \]
    This leads to
    \[
    \sqrt{\frac{\sigma^2+H_{N+1}(1+\frac{N+1}{u(\sigma)^2-1})}{u(\sigma)^2-1}}\le \sqrt{\frac{\sigma^2+\log(1+u(\sigma)^2-1)}{u(\sigma)^2-1}}.
    \]
    By the definition of $u(\sigma)$, we have $u(\sigma)^2-1=\sigma^2+2\log u(\sigma)$, so that the right-hand-side equals $1$. This results in $u_N'\le u(\sigma)$.
    This holds for the sequence
    \begin{equation}\label{eq:alpha-prime1}
    \alpha_k'=\frac{R}{L}\frac{N-k}{(N+1)^{3/2}}\frac{y_k'}{y_0'u_N'}.
    \end{equation}
    We compute
    \[
    \frac{y_k'}{y_0'}=\frac{1}{1-ky_0'}=\frac{N+1}{N+1-k(1-u(\sigma)^{-2})}=\frac{(N+1)u(\sigma)^2}{(N+1)u(\sigma)^2-(u(\sigma)^2-1)k}.
    \]
    This results in
    \[
    \frac{y_k'}{y_0'u_N'}=\frac{(N+1)u(\sigma)}{(N+1)u(\sigma)^2-(u(\sigma)^2-1)k}\sqrt{\frac{\sigma^2+2\log u(\sigma)}{\sigma^2+H_{N+1}\left(1+\frac{N+1}{u(\sigma)^2-1}\right)}}.
    \]
    The expression for the step sizes is obtained by substituting this quantity into~\eqref{eq:alpha-prime1}.

    We now show that the quantity in the square root converges to $1$. Let
    \[
    q=\frac{\sigma^2+H_{N+1}\left(1+\frac{N+1}{u(\sigma)^2-1}\right)}{\sigma^2+2\log u(\sigma)}=1+\frac{H_{N+1}\left(1+\frac{N+1}{u(\sigma)^2-1}\right)-2\log u(\sigma)}{\sigma^2+2\log u(\sigma)},
    \]
    so that the quantity in the square root is $q^{-1/2}$.
    Using the same integral bound as before, it follows that $q\le 1$. To lower-bound $q$, we use the other integral bound to obtain
    \begin{eqnarray*}
    H_{N+1}\left(\frac{N+u^2}{u^2-1}\right)
    &\ge \log\left(1+\frac{N+1}{\frac{N+u^2}{u^2-1}}\right)
    &=\log\left(1+(u^2-1)\frac{N+1}{N+u^2}\right)\\
    &=\log\left(u^2-\frac{(u^2-1)^2}{N+u^2}\right)
    &=2\log u+\log\left(1-\frac{(u^2-1)^2}{u^2(N+u^2)}\right)\\
    &\ge 2\log u+\log\left(\frac{N+1}{N+u^2}\right)
    &=2\log u-\log\left(1+\frac{u^2-1}{N+1}\right)\\
    &\ge 2\log u-\frac{u^2-1}{N+1}.
    \end{eqnarray*}
    This leads to the following bound
    \[
    q-1\ge -\frac{\frac{u^2-1}{N+1}}{\sigma^2+2\log u}=-\frac{1}{N+1},
    \]
    where we used $\sigma^2+2\log u=u^2-1$ in the last step.
    We conclude that
    \[
    1\le q^{-1/2}\le \sqrt{\frac{N+1}{N}}\le 1+\frac{2}{N}.
    \]
\end{appendixproof}

The bounds in Equation~\eqref{eq:negligible-factor} show that the square root factor in the expression of $\alpha'_k$ is nearly negligible.
Comparing the performance guarantee from~\cref{lem:admissible} to that of~\cref{prop:u-L-bound}, we see that the $N$-dependent quantity $u_N^{\mathbb L}(\sigma)$ is replaced by $u(\sigma)$, which does not depend on $N$.
Finally, we remark that the function $u$ can be represented with the help of the \emph{Lambert W-function}~\citep{Corless_Gonnet_Hare_Jeffrey_Knuth_1996} as
\[
u=\sqrt{-W_{-1}(-e^{-1-\sigma^2})},
\]
where $W_{-1}(z)$ is be the negative real solution in $we^{w}=z\in[-e^{-1},0)$.
However, in what follows we find it more convenient to derive properties directly from its implicit representation in Equation~\eqref{eq:lambda-implicit}.
\begin{lemma}\label{lem:lambda}
    The function $u(\sigma)$ defined by Equation~\eqref{eq:lambda-implicit} is convex, satisfies the bounds
    \[
    \max\left\{\sqrt2\sigma,\sigma^2+\log(1+\sigma^2)\right\}\le u(\sigma)^2-1\le 
    \sqrt2\sigma+\sigma^2,
    \]
    and has asymptotics
    \[
    u(\sigma)=\begin{cases}
        \sigma+\bigO{\sigma^{-1}\log\sigma}&\text{ as }\sigma\to\infty,\\
        1+\sqrt{\frac12}\sigma+\bigO{\sigma^2},&\text{ as }\sigma\to0.
    \end{cases}
    \]
\end{lemma}
\begin{appendixproof}[Proof of~\cref{lem:lambda}]
    To prove convexity, we take the first two derivatives of~\eqref{eq:lambda-implicit} w.r.t. $\sigma$ and obtain
    \[
    2\sigma=u'\cdot (2u-2/u)\Rightarrow u'=\frac{u\sigma}{u^2-1},
    \]
    and
    \[
    2=u''\cdot (2u-2/u)+2(u')^2(1-u^{-2})\Rightarrow u''=\frac{u}{(u^2-1)^2}\left(u^2-1-\sigma^2\right),
    \]
    which is nonnegative since $u^2-1-\sigma^2=2\log u\ge0.$ 
    The lower bound $\sqrt2\sigma$ is derived using
    \[
    u^2-1-\sigma^2=\log(1+u^2(\sigma)-1)\ge u(\sigma)^2-1-\frac12(u(\sigma)^2-1)^2,
    \]
    which results in $u^2-1\ge\sqrt{2\sigma^2}.$ To derive the other lower bound, we write
    \begin{equation}\label{eq:u-bootstrap}
    u^2=1+\sigma^2+\log(u^2).
  \end{equation}
    From $u^2\ge1$, we deduce
    \[
    u^2\ge 1+\sigma^2.
    \]
    Substituting this bound into~\eqref{eq:u-bootstrap} yields
    \[
    u^2\ge1+\sigma^2+\log(1+\sigma^2).
    \]
    For the upper bound, we write
    \begin{eqnarray*}
        u(\sigma)^2-1
        &=\int_0^\sigma 2u(s)\cdot u'(s)ds
        &=\int_0^\sigma 2s\frac{u(s)^2}{u(s)^2-1}ds
        =\int_0^\sigma 2s\left(1+\frac{1}{u(s)^2-1}\right)ds\\
        &\le\int_0^\sigma (2s+\sqrt2)ds
        &=\sigma^2+\sqrt2\sigma,
    \end{eqnarray*}
    where we used the bound $u(\sigma)^2-1\ge\sqrt2\sigma$.
    
    We now derive the asymptotics. Note that $u(\sigma)\to\infty$ as $\sigma\to\infty$, so that the right-hand-side of~\eqref{eq:lambda-implicit} is dominated by $u^2$, which leads to the $u(\sigma)\sim\sigma$ asymptotics. For $\sigma\to0$, the bounds yield the desired asymptotics, after using
    \[
    u(\sigma)=\sqrt{1+\sqrt{2}\sigma+\bigO{\sigma^2}}=1+\sqrt{1/2}\sigma+\bigO{\sigma^2}.
    \]
\end{appendixproof}

The following theorem shows that the essential behaviour of the bounds introduced so far, i.e., $u_N^{\mathbb S},u_N^{\mathbb L}$ and $\ell_N$, are all captured by the function $u$ defined in Equation \eqref{eq:lambda-implicit}.
\begin{theorem}\label{thm:squeeze}
    For any $N$ and $\sigma=\gamma/L\in[0,\sqrt{N}]$, the following inequalities hold
    \[
    \left(1-\frac52\frac{\log (N+1)}{N}\right)u(\sigma)\le \ell_N(\sigma)\le u_N^{\mathbb S}(\sigma)\le u_N^{\mathbb L}(\sigma)\le u(\sigma).
    \]
\end{theorem}
\begin{appendixproof}[Proof of~\cref{thm:squeeze}]
The inequalities $\ell_N(\sigma)\le u_N^{\mathbb S}(\sigma)\le u_N^{\mathbb L}(\sigma)\le u(\sigma)$ follow from the fact that $\ell_N$ lower bounds the performance of a class of methods that include the minimizer of $u^{\mathbb S}_N$, while $u^{\mathbb L}_N$ and $u$ are obtained by adding additional constraints to the minimization problem. It therefore suffices to prove that
    \[
    \frac{\ell_N(\sigma)}{u(\sigma)}\ge\left(1-\frac{\log N}{N}-\frac{1+2\sqrt{N}}{N^2}\right).
    \]
    We can rewrite the $\nu$-constraint to
\begin{eqnarray}
\sigma^2
&=\nu \sum_{i=0}^{N-1}  \frac{(N-i)\nu}{1-\nu+(N-i)\nu}\nonumber
&=N\nu-\nu\sum_{i=0}^{N-1}  \frac{1-\nu}{1-\nu+(N-i)\nu}\nonumber\\
&=N\nu-(1-\nu)\sum_{k=1}^{N}  \frac{1}{\frac1\nu-1+k}
&=N\nu-(1-\nu)H_{N}(1/\nu)\label{eq:nu-implicit},
\end{eqnarray}
where $H_m(a)$ the generalized harmonic number. 
We study the asymptotics of $H_m(a)$. Integral bounds result in
\[
\log\left(1+\frac{m}{a}\right)\le H_m(a)\le \log\left(1+\frac{m}{a-1}\right).
\]
This leads to
\[
(1-\nu)\log(1+N\nu)\le N\nu-\sigma^2\le(1-\nu)\log\left(1+\frac{N\nu}{1-\nu}\right),
\]
For the right-hand-side, we write $\frac1{1-\nu}=1+\frac{\nu}{1-\nu}$ and use concavity to bound
\[
\log\left(1+\frac{N\nu}{1-\nu}\right)\le \log\left(1+N\nu\right)+\frac{\frac{N\nu^2}{1-\nu}}{1+N\nu},
\]
so that
\[
N\nu-\log(1+N\nu)-\sigma^2\in\left[-\nu\log(1+N\nu),\frac{N\nu^2}{1+N\nu}-\nu\log(1+N\nu)\right].
\]
Substituting $\ell_N(\sigma)=\sqrt{1+N\nu}$, or $\nu=\frac1N(\ell_N^2-1)$ yields
\[
\ell_N^2-1-2\log\ell_N-\sigma^2\in\left[-2\frac{\ell_N^2-1}N\log\ell_N,\frac{\ell_N^2-1}{N}\left(\frac{\ell_N^2-1}{\ell_N^2}-2\log\ell_N\right)\right].
\]
In this, we recognize the definition of $u(\sigma)$ given in~\eqref{eq:lambda-implicit}.
    This leads to the bounds
    \begin{equation}\label{eq:l-bounds}
    u\left(\sigma\sqrt{1-2\frac{\ell_N^2-1}{N\sigma^2}\log\ell_N}\right)\le\ell_N\le u\left(\sigma\sqrt{1-\frac{\ell_N^2-1 }{N\sigma^2}\left(2\log\ell_N-\frac{\ell_N^2-1}{\ell_N^2}\right)}\right).
    \end{equation}
    Applying $\log(1+x)\le\frac{x}{x+1}$ to $2\log\ell_N=\log(1+\ell_N^2-1)$ already tells us that the right-hand-side is at most $u(\sigma)$.
    Applying $\ell_N(\sigma)\le u(\sigma)$ to the left-hand-side leads to
    \[
    \ell_N(\sigma)\ge u\left(\sigma\sqrt{1-2\frac{u(\sigma)^2-1}{N\sigma^2}\log u(\sigma)}\right).
    \]
    Using the bound $\sqrt{1-x}\ge1-x$, we obtain
    \[
    \ell_N(\sigma)\ge u\left(\sigma-2\frac{u(\sigma)^2-1}{N\sigma}\log u(\sigma)\right).
    \]
    Next, we use convexity of $u(\sigma)$ to write 
    \[
    u\left(\sigma-2\frac{u(\sigma)^2-1}{N\sigma}\log u(\sigma)\right)\ge u(\sigma)-2\frac{u(\sigma)^2-1}{N\sigma}\log u(\sigma)\cdot u'(\sigma).
    \]
    Next, we substitute
    \[
    u'(\sigma)=\frac{u(\sigma)\cdot\sigma}{u(\sigma)^2-1}
    \]
    to obtain
    \[
    \ell_N(\sigma)\ge u(\sigma)-\frac{2u(\sigma)\log u(\sigma)}{N}=u(\sigma)\cdot\left(1-\frac{\log( u(\sigma)^2)}{N}\right).
    \]
    Finally, the bound $u(\sigma)^2\le1+2\sigma +\sigma^2$ from~\cref{lem:lambda} and $\sigma\le\sqrt{N}$ yield $\log u(\sigma)^2\le \log(1+2\sqrt{N} +N)\le\log (N+1)+\frac{2\sqrt{N}}{N+1}\le (1+(\log2)^{-1})\log(N+1)$ for $N\ge1$, where the last step follows from the fact that $2\sqrt{N}/((N+1)\log(N+1))$ is decreasing and equal to $(\log2)^{-1}$ for $N=1$. The desired bound follows from the fact that $1+(\log2)^{-1}<\frac52$.
\end{appendixproof}


\cref{thm:squeeze} tells us that the performance of a subgradient method with associated analytic conic combination $\alpha'$ is asymptotically equivalent to a subgradient method with the associated conic combination $\alpha^{\mathbb L}$ proposed in Proposition \ref{prop:u-L-bound}.
It furthermore analytically shows that the relative worst-case suboptimality gap depicted in Figure \ref{fig:rel_gap} is small as indeed
\(
  (\ref{eq:relative-gap}) \leq \max_{\sigma \in [0, \sqrt{N}]} \tfrac{(u(\sigma)-\ell_N(\sigma))}{\ell_N(\sigma)}\leq 5/2\log(N+1)/N/(1-5/2\log(N+1)/N)=\mc O(\log(N)/N).
\)
The next result shows that the step sizes themselves are also asymptotically equivalent for moderately small $\sigma$.

\begin{lemma}\label{lem:asymptotics-moderate-sigma}
    For $\sigma\ll N^{1/4},$ the optimal conic combination $\alpha^{\mathbb L}$ in \cref{prop:u-L-bound} are asymptotically equivalent to the analytic conic combination $\alpha'$ from~\cref{lem:admissible}. That is,
    \begin{equation}\label{eq:alpha-opt-asymp-i}
    \alpha_k^{\mathbb L}\sim \frac{R(N-k)}{L(N+1)^{3/2}}\cdot\frac{u(\sigma)}{u(\sigma)^2-(u(\sigma)^2-1)\frac{k}{N+1}},
    \end{equation}
    where $u(\sigma)\ge1$ is defined implicitly in Equation~\eqref{eq:lambda-implicit}.
\end{lemma}
\begin{appendixproof}[Proof of~\cref{lem:asymptotics-moderate-sigma}]
    We will use the bounds of~\cref{lem:x-seq-sum} to write
    \[
    \sum_{i=1}^{2^N}(k)_{i-1}y_0^i\le y_k\le \sum_{i=1}^{2^N}k^{i-1}y_0^i.
    \]
    For $y_0<k^{-1}$, we can further bound the right-hand-side to an infinite sum, which results in the known power series
    \[
    y_k\le \sum_{i=1}^{\infty}k^{i-1}y_0^i=\frac{y_0}{1-ky_0}.
    \]
    For the lower bound, we will first lower-bound the falling factorial using an integral bound and an expansion of the logarithm:
    \begin{eqnarray*}
    \log(n)_r&=\sum_{x=n-r+1}^n\log x
    &\ge\int_{n-r}^n\log x dx\\
&=n\log n-(n-r)\log(n-r)-r
&=r\log n-(n-r)\log\left(1-\frac rn\right)-r\\
&\ge r\log n-\frac{r^2}{n},
    \end{eqnarray*}
    so that
    \[
    (n)_r\ge n^r\cdot e^{-r^2/n}\ge n^r\cdot\left(1-\frac{r^2}n\right).
    \]
    Notice that $(k)_{i-1}=0$ for $i>k+1$. Hence,
    \[
    y_k\ge \sum_{i=1}^{2^N}(k)_{i-1}y_0^i=\sum_{i=1}^\infty(k)_{i-1}y_0^i\ge \sum_{i=1}^\infty k^{i-1}\cdot\left(1-\frac{(i-1)^2}{k}\right)y_0^i=\frac{y_0}{1-ky_0}-\frac{y_0}{k}\sum_{i=0}^\infty i^2y_0^i,
    \]
    for $y_0<k^{-1}$. Using $\sum_{i=0}^\infty i^2w^i,=\frac{w(1+w)}{(1-w)^3}$ for $w=ky_0<1$, we obtain the lower bound
    \[
    y_k\ge \frac{y_0}{1-ky_0}\left(1-\frac{y_0(1+ky_0)}{(1-ky_0)^2}\right),
    \]
    which is asymptotically equivalent to the upper bound whenever 
    \begin{equation}\label{eq:y_k-condition}
        \frac{y_0}{(1-ky_0)^2}\to0.
    \end{equation}
    To see for what $\sigma$ this holds, we inspect the asymptotics of~\eqref{eq:y0-implicit}. 
    We rewrite
    \[
    \sigma^2=y_0S_N'(y_0)-S_N(y_0)=\sum_{i=1}^{2^N}(i-1)\cdot s_{N,i}\cdot y_0^i,
    \]
    where $s_{N,i}=\sum_{k=0}^Na_{k,i}$ is the $i$-th coefficient of $S_N(y_0).$ \cref{lem:x-seq-sum} proves the bounds
    \[
    \frac1i(N+1)_{i}\le s_{N,i}\le \frac1i(N+1)^{^i}.
    \]
    This leads to the bounds
    \[
    \sigma^2\le \sum_{i=1}^\infty\left(1-\frac1i\right)\cdot (N+1)^i\cdot y_0^i=\frac{(N+1)y_0}{1-(N+1)y_0}+\log(1-(N+1)y_0),
    \]
    and
    \begin{align*}
    \sigma^2
    &\ge \sum_{i=1}^\infty\left(1-\frac1i\right)\left(1-\frac{i^2}{N+1}\right)\cdot (N+1)^i\cdot y_0^i\\
    &=\frac{(N+1)y_0}{1-(N+1)y_0}+\log(1-(N+1)y_0)-(N+1)y_0^2\sum_{i=1}^\infty i(i-1)(N+1)^{i-2} y_0^{i-2},
    \end{align*}
    where we were again able to extend the sum to infinity because $(N+1)_{i}=0$ for $i>N+1$. Using 
    \[
    \sum_{i=2}^\infty i(i-1)w^{i-2}=\frac{d^2}{dw^2}\sum_{i=0}^\infty w^i=\frac{d^2}{dw^2}\frac1{1-w}=\frac{2}{(1-w)^3},
    \]
    we obtain
    \[
    \sigma^2=\frac{(N+1)y_0}{1-(N+1)y_0}+\log(1-(N+1)y_0)+\bigO{\frac{(N+1)y_0^2}{(1-(N+1)y_0)^3}}.
    \]
    Introduce $z_0=\frac{(N+1)y_0}{1-(N+1)y_0}$, then
    \[
    \sigma^2=z_0-\log(1+z_0)+\bigO{\frac{z_0^2(1+z_0)}{N+1}}.
    \]
    We look for regimes where the error term is negligible, so that $z_0\sim z(\sigma)$, where $z(\sigma)$ is the positive solution in
    \begin{equation}\label{eq:z-implicit}
        \sigma^2=z-\log(1+z).
    \end{equation}
    Notice that $z_0-\log(1+z_0)\sim z_0$ for large $z_0$ and $z_0-\log(1+z_0)\sim \frac12z_0^2$ for small $z_0$. Hence, for $\sigma\to\infty$, we need $z_0\ll \sqrt{N}$ to get $z_0\sim \sigma^2$. This occurs whenever $\sigma\ll N^{1/4}$.
    For $z_0\to0$, we need $\frac{z_0^2(1+z_0)}{N+1}\ll z_0^2$ which always holds since $N\to\infty.$
    This tells us that whenever $\sigma\ll N^{1/4}$, we have $z_0\sim z(\sigma)$, and 
    \[y_0=\frac1{N+1}\frac{z_0}{z_0+1}\sim\frac1{N+1}\frac{z(\sigma)}{z(\sigma)+1}.\]

    Returning to the condition~\eqref{eq:y_k-condition}, we note that
    \[
    \frac{y_0}{(1-ky_0)^2}<\frac{y_0}{(1-(N+1)y_0)^2}=\frac{z_0(1+z_0)}{N+1},
    \]
    which indeed vanishes since $z_n\ll\sqrt{N}$. This tells us that $y_k\sim\frac{y_0}{1-ky_0}$ for $\sigma\ll N^{1/4}$.
    Finally, from~\cref{thm:squeeze}, it follows that $u_N^{\mathbb L}(\sigma)\sim u(\sigma)$, so that
    \[
    \alpha_k^{\mathbb L}\sim \alpha_k^*\cdot \frac{1}{u(\sigma)\cdot(1-ky_0)},
    \]
    where $\alpha_k^*$ are the optimal step sizes for the noiseless case given in~\eqref{eq:alpha-noiseless}. The result follows after substituting
    \[
    y_0\sim\frac1{N+1}\frac{u(\sigma)^2-1}{u(\sigma)^2}.
    \]
  \end{appendixproof}

  The expression~\eqref{eq:alpha-opt-asymp-i} helps explain the shape of the step size sequences depicted in Figure~\ref{fig:steps}.
  Let us reindex the iterations using $\theta\in[0,1)$ via $k=\lfloor N\theta\rfloor$ and rescale it appropriately
  $\alpha^{\mathbb L}(\theta)= \alpha^{\mathbb L}_{\lfloor N\theta\rfloor}$ for $\theta\in[0,1)$.
  Recall that in the special case $\sigma=0$, we have that $\alpha^{\mathbb L}_k(\sigma)$ coincides with the conic combination in Equation (\ref{eq:alpha-noiseless}) and hence
  \(
    \alpha^{\mathbb L}(\theta) \sim \tfrac{R}{L\sqrt{N+1}}  (1-\theta)
  \)
  corresponding to the linear line in Figure \ref{fig:steps}.
  More generally, from Lemma \ref{lem:asymptotics-moderate-sigma} it follows that
  \[
    \alpha^{\mathbb L}(\theta)\sim \frac{R}{L\sqrt{N+1}} \cdot \frac{1-\theta}{u(\sigma)(1-(1-u(\sigma)^{-2})\cdot\theta)}
  \]
  which is depicted as the dotted line in Figure \ref{fig:steps} for $\sigma=5$.
  From the previous we also deduce that the (near) optimal subgradient methods for $\sigma>0$ become more aggressive than the noiseless subgradient method associated with the conic combination in Equation (\ref{eq:alpha-noiseless}) in the regime
  \(
    \theta>\tfrac{u(\sigma)}{(1+u(\sigma))}.
  \)


The bound given in~\cref{thm:main-upper-bound} improves the trivial bound $RL$ whenever $u_N^{\mathbb L}(\sigma)<\sqrt{N+1}$. Using the upper bound $u_N^{\mathbb L}(\sigma)\le u(\sigma),$ we can guarantee that for
\[
\sigma^2<N-\log(N+1),
\]
it holds that $u_N^{\mathbb L}(\sigma)\le u(\sigma)<\sqrt{N+1}$. 

\section{Discussion}

In this paper we advance three subgradient methods (see Corollary \ref{cor:u-S-bound}, Proposition \ref{prop:u-L-bound}, and Lemma \ref{lem:admissible}, respectively) each of which, as implied by Theorem \ref{thm:squeeze}, enjoy near-optimal performance in terms of their relative worst-case suboptimality gap when minimizing nonsmooth convex functions $f$ faced with adversarial subgradient corruption.
Each of these subgradient methods can hence be regarded as an inexact generalization of the classical subgradient method (\ref{eq:classical-stepsize}).
In this section, we discuss several possible extensions of this work.

\paragraph{Projected Subgradient Methods.}
In the presence of a convex constraint $x\in C$, projected subgradient methods 
\begin{equation}
  \label{eq:proj_gd}
  x_{k+1} = P_C(x_k-h_k \tilde g_k)
\end{equation}
are considered instead where $P_C(y) = \arg\min_{x\in C} \norm{y-x}^2$ denotes the projection operator. 
A well known property of the projection operator guarantees that iteration \eqref{eq:proj_gd} can be represented equivalently as
\begin{equation}
  \label{eq:proj_prop}
  \norm{x_{k+1}-y}^2 \leq \norm{x_k-h_k \tilde g_k - y}^2 \quad \forall y\in C.
\end{equation}
In the classical proof \citep{boyd2003subgradient, lan2020first} of Equation (\ref{eq:classical-upper-bound}), it is established that the claimed performance guarantee holds  for any iteration scheme which satisfies merely
\begin{equation}
  \label{eq:proj_prop_relax}
  \norm{x_{k+1}-x^\star}^2 \leq \norm{x_k-h_k \tilde g_k - x^\star}^2 \quad \forall k\in [0, \dots, N-1].
\end{equation}
As performance lower bound are not affected by auxiliary constraints (the restriction $C$ may indeed be chosen as $X$), it follows that through projection the subgradient methods remains worst-case optimal even when facing convex restrictions on which projection is simple.
Extending the results in this paper to work with convex restrictions may at first glance seem daunting as our upper performance bound result in Lemma \ref{lem:bound} uses in Equation (\ref{eq:linearity-iterations}) the affine relation $x_{k+1} = x_0+ \sum_{j=0}^kh_j\tilde g_j$ between iterates and subgradients which clearly fails to hold if $C\neq X$.
However, inspired by Equation \eqref{eq:proj_prop} and akin to \eqref{eq:proj_prop_relax}, it is rather straightforward (though very tedious) to show that the result in Proposition \ref{thm:main-upper-bound} remains valid for any subgradient iterations which satisfies
\begin{equation}
  \label{eq:proj_prop_relax_2}
  \norm{x_{k+1}-y}^2 \leq \norm{x_k-h_k \tilde g_k - y}^2 \quad \forall y \in \{x^\star, x_0, \dots, x_{N}\}, ~\forall k\in [0, \dots, N-1].
\end{equation}
Hence, although we chose to omit the details of this generalization as not to negatively affect the exposition of this paper, all results in this paper still hold when $f$ is restricted to $C$ and where the projection iteration suggested in Equation \eqref{eq:proj_gd} is used instead.

\paragraph{Universal Subgradient Methods.}

The near-optimal subgradient methods identified here depend on the problem parameters $L$ and $\sigma=\gamma/L$.
In practice, finding good values for these parameters may prove challenging.
In the noiseless case, this can be addressed by \emph{normalized subgradient descent}, where $\|g_k\|$ is substituted for $L$ in the step sizes~\eqref{eq:classical-stepsize} or~\eqref{eq:zamani-stepsize}, which enjoys the same optimal~\citep{boyd2003subgradient,zamani2023exact} performance guarantee \eqref{eq:classical-upper-bound} while being adaptive to $L$. 
Generalizing this observation to inexact subgradients presents a promising direction of research but has to face the problem that the normalization $\norm{g_k}$ is not observed directly, but the corrupted version $\norm{\tilde g_k}$ will have to be used instead. We note, however, that the lower bound~\cref{cor:lower} does hold for normalized subgradient descent.

Finally, the near-optimal subgradient methods dependend on the power of the adversary as characterized by $\sigma$. Unlike the Lipschitz constant $L$, there does not appear to be a candidate estimator for this parameter. This is a common challenge that is faced in adversarial environments.  Typically, the value of $\sigma$ necessarily reflects a certain amount of domain expertise which is to be taken at face value.

\paragraph{Smooth Convex Optimization.}
In this work, we focused on \emph{nonsmooth} optimization. However, the PEP approach that is at the core of \cref{lem:bound} and \cref{thm:lower-matrix} has been extended to smooth (strongly) convex functions by \cite{taylor_smooth_2017, deklerk2017worst}.
Gradient methods with optimal worst-case suboptimality $\mc O(N^{-2})$ make use of \emph{momentum}~\citep{kim_optimized_2016, nesterov1983method} whereas a simple gradient method with constant step size suffer a worst-case suboptimality of $\mc O (N^{-1})$.
By allowing nonconstant step sizes, as we do here, recent concurrent work by \citet{grimmer2024provably} and \citet{altschuler2025acceleration} prove that worst-case suboptimality can be improved to $\mc O(N^{-1-\delta})$ with $\delta\approx 0.27$ which is conjectured to be unimprovable without momentum.
We expect this richer landscape of smooth convex optimization to translate to a significantly more challenging analysis when studying the impact of adversarial noise.
It has been observed that momentum methods are more sensitive to error accumulation~\citep{devolder2014first,liu2024nonasymptotic} than simple gradient methods. This suggests that optimal methods for nonsmooth optimization with corrupted gradients may rely less on momentum as the level of noise increases.


\bibliography{main.bib}

\end{document}